\documentclass[12pt]{article}
\usepackage{graphicx, amssymb, amsthm, amsfonts, latexsym, epsfig, amsmath, amscd, amsbsy}

\theoremstyle{definition}
\newtheorem{theorem}{Theorem}[section]
\newtheorem{lem}[theorem]{Lemma}
\newtheorem{cor}[theorem]{Corollary}
\newtheorem{pro}[theorem]{Proposition}
\newtheorem{df}[theorem]{Definition}
\newtheorem{rem}[theorem]{Remark}
\newtheorem{ex}[theorem]{Example}

\numberwithin{equation}{section}

\def\N{{\mathbb N}}
\def\Z{{\mathbb Z}}
\def\C0{{\mathfrak C}}

\def\Int{\mathop\mathrm{Int}}

\def\ILs{\underleftarrow\lim([0,1],T_{s})}
\def\ILcores{\underleftarrow\lim([c_2, c_1],T_{s})}

\def\ILsp{\underleftarrow\lim([0,1],T_{s'})}

\def\chain{{\mathcal C}}

\newcommand{\bd}{\partial}
\def\ie{{\em i.e.,}\ }
\def\eg{{\em e.g.}\ }
\def\eps{\varepsilon}
\def\diam{\mbox{diam}}

\def\mesh{\mbox{mesh}}
\def\width{\mbox{width}}

\def\Int{\mathop\mathrm{Int}}
\def\Cl{\mathop\mathrm{Cl}}

\begin{document}

\title{The Ingram Conjecture}
\author{M.~Barge, H.~Bruin, S.~\v{S}timac\thanks{Supported in part by NSF 0604958 and in part by the
MZOS Grant 037-0372791-2802 of the Republic of Croatia.}
\thanks{The authors thank the Mathematisches Forschungsinstitut Oberwolfach for its hospitality during
the Research in Pairs Programme January 11-24, 2009.
HB also thanks Delft University of Technology, where this paper was
largely completed.}}

\maketitle

\begin{abstract}
We prove the Ingram Conjecture, \ie we show that the inverse limit spaces of every two tent maps with
different slopes in the interval $[1, 2]$ are non-homeomorphic. Based on the structure obtained from the proof, we
also show that every self-homeomorphism of the inverse limit space of the tent map is pseudo-isotopic,
on the core, to some power of the shift homeomorphism.
\end{abstract}



\baselineskip=18pt

\section{Introduction}\label{sec:intro}

Apart from their interest within continuum theory,
inverse limit spaces play a key role in the description of
uniformly hyperbolic attractors \cite{W1,W2}, global `H\'enon-like' strange
attractors \cite{BH} and the structure emerging from homoclinic tangencies
in dynamical systems \cite{BaDi2}. They find further use in the area
of (substitution) tiling spaces \cite{AP} which, in some cases, are covering spaces of the type of
inverse limit spaces with which we are concerned with in this paper; namely,
those with a single tent map $T_s:[0,1] \to [0,1]$,
$x \mapsto \min\{ sx, s(1-x) \}$ as bonding map.
Such inverse limit spaces can be embedded in the plane as global attractors of homeomorphisms
\cite{Mis, Sz, Bembeddings} and immersed in the plane as global attractors of
skew product maps  \cite{HRS}.

Inverse limit spaces are notoriously difficult to classify.
In this paper, we solve in the affirmative the classification problem
known as the Ingram Conjecture:
\begin{theorem}[\textbf{Ingram Conjecture}]\label{thm:Ingram}
If $1 \leq s < s' \leq 2$, then the corresponding inverse limit spaces
$\ILs$ and $\ILsp$ are non-homeo\-mor\-phic.
\end{theorem}
This is the main outstanding conjecture regarding dynamics on continua, dating back to at least the early nineties.
In the ``Continua with the Houston problem book'' in 1995
\cite[page 257]{Ingram}, Ingram writes
\begin{quote}
The [...] question was asked of the author by Stu Baldwin
at the summer meeting of the AMS at Orono, Maine, in 1991...
There is a related question which the author has considered to be of
interest for several years. He posed it at a problem session at the
1992 Spring Topology Conference in Charlotte for the special case
({\em that the critical point has period}) $n=5$.
\end{quote}
It is clear that if two interval maps are topologically conjugate, then
their inverse limit spaces are homeomorphic. Thus it may be more natural
to ask the question for the `fuller' logistic family $f_a(x) = ax(1-x)$, $a \in [0,4]$.
It is well-known \cite{MT} that each logistic map is semi-conjugate to a
tent-map $T_s$ with $s = \exp(h_{top}(f_a))$ (provided $h_{top}(f_a) > 0$),
but the logistic maps
contain infinitely renormalizable maps as well as maps
with periodic attractors, phenomena that are ruled out in the
(uniformly expanding) setting of tent maps with slope $s > 1$.
The effect of renormalization (\ie the existence of periodic
intervals for period $> 1$) on the structure of the inverse limit space
is well-understood, see \cite{BaDi3}: it produces proper subcontinua
that are periodic under the shift homeomorphism, that are homeomorphic with
the inverse limit space of the renormalized map.
Therefore the solution of the Ingram Conjecture also shows that
every pair of logistic maps that are  non-conjugate on their non-wandering
sets have non-homeomorphic inverse limit spaces.

There have been several partial results to the Ingram Conjecture,
\eg Barge and Diamond \cite{BaDi1}, which solved the period $n=5$ case,
and \cite{SV, Bfinite}.
Complete solutions were obtained when the critical
point is periodic by Kailhofer \cite{Kail2} (see also \cite{BJK}), or has finite orbit
by \v{S}timac \cite{Stim}.
More recently, the case where
the critical point is non-recurrent was solved in \cite{RS}. Further results
that classify certain features of inverse limit spaces of tent maps with
non-periodic recurrent critical orbits were obtained in \eg
\cite{BB, Rain, Bsubcontinua}.

Our solution to the Ingram Conjecture gives more information about
the set of self-homeomorphisms on $\ILs$: we show that
any such homeomorphism behaves like an iterate of the shift homeomorphism $\sigma$.

The {\em critical point} $\frac12$ of $T_s$ is denoted by $c$, and
we write $c_i = T^i(c)$. Although $T_s$ is defined on $[0,1]$, there
is a forward invariant interval $[c_2, c_1] = [s(1-s/2),s/2]$,
called the {\em core}, on which $T_s$ is surjective. We call
$\ILcores$ the {\em core of the inverse limit space}. The space $\ILs$ is
the union of the core of the inverse limit and a ray $\C0$ converging onto
it.

Recall that the \emph{composant} of $x \in X$ is defined as the
union of all proper subcontinua of $X$ containing $x$. For $1\le s<2$,
$\ILs$ has only two composants: $\C0$ and $\ILs \setminus
\{ (\dots,0,0,0) \}$. But for $s>\sqrt{2}$, $\ILcores$ is indecomposable and hence has uncountably many pairwise disjoint composants, each of which is dense. If $s>\sqrt{2}$ and the orbit of $c$ is finite, the composants of
$\ILcores$ are the same as the arc-components. Otherwise, the composants can be very complicated.
For $1<s\le \sqrt{2}$, the core has just two composants that overlap in a single arc-component.

\begin{theorem}\label{thm:pseudo}
Given $s \in [1, 2]$,  for every homeomorphism $h:\ILs \circlearrowleft$, there is an $R \in \Z$
such that $h$, restricted to the core $\ILcores$, is pseudo-isotopic to $\sigma^R$, \ie it permutes the
composants of the core of the inverse
limit in the same way as $\sigma^R$.
\end{theorem}

Our proof of the Ingram Conjecture relies on the properties of so-called link-symmetric arcs in the
composant $\C0$ of $\ILs$ containing the endpoint $\alpha := (\dots,0,0,0)$. Inverse limit spaces are
chainable, and w.r.t.\ natural chains, a homeomorphism $h:\ILsp \to \ILs$ maps link-symmetric
arc to link-symmetric arcs. From this we derive that maximal link-symmetric arcs in $\ILsp$
centered at so-called snappy points $s'_i$ map to link-symmetric arcs centered at snappy points
$s_{i+M} \in \ILs$ for some $M \in \Z$ and all sufficiently large $i \in \N$.

This in turn implies that $h$ maps so-called $q$-points close to $p$-points, while `translating' their
levels by a fixed number $M$. This shows that $h$ effectively fixes the folding pattern of the zero-composant,
with the Ingram Conjecture as an easy consequence. Additional arguments show that every self-homeomorphism
of $\ILs$, when restricted to the core, is \emph{pseudo-isotopic} to a power $\sigma^R$ of the shift for some $R \in \Z$.

We give the basic definitions in the next section. In Section~\ref{sec:lemmas} we investigate the lengths of
maximal link-symmetric arcs, leading in Section~\ref{sec:link-symmetry} to the proof that a homeomorphism
between two unimodal inverse limit spaces induces a shift of indices of snappy points, and more generally,
acts as a shift on the levels of $q$-points and $p$-points. This leads to the proof of the Ingram Conjecture.
Finally, in Section~\ref{sec:Pseudo}, we prove the remaining results on pseudo-isotopy.

\section{Definitions}\label{sec:def}

Let $\N := \{ 1, 2, \dots \}$ and  $\N_0 := \{ 0, 1, 2, \dots \}$.
Let $T_s:[0,1] \to [0,s/2]$, $T_s(x) = \min\{ sx, s(1-x)\}$ be the
tent map with slope $s \in [1,2]$ and critical point $c = \frac12$.
Write $c_i = c_i(s) := T_s^i(c)$, so in particular $c_1 = \frac{s}{2}$ and $c_2
= s(1-\frac{s}{2})$.

The inverse limit space $\ILs$ is the collection of backward orbits
\[
 \{ x = (\dots, x_{-2}, x_{-1}, x_0) :  T_s(x_{i-1}) = x_i \in [0,s/2]
\text{ for all } i \leq 0\},
\]
equipped with metric $d(x,y) = \sum_{n \le 0} 2^n |x_n - y_n|$ and {\em induced $($or shift$)$ homeomorphism}
$\sigma = \sigma_s$ given by
\[
\sigma(\dots, x_{-2}, x_{-1}, x_0) = (\dots, x_{-2}, x_{-1}, x_0, T_s(x_0)).
\]
Let $\pi_p: \ILs \to [0, 1]$, $\pi_p(x) = x_{-p}$, be the $p$-th projection map. Since $T_s$ fixes $0$,
$\ILs$ contains the endpoint $\alpha  := (\dots, 0,0,0)$. The composant of $\ILs$ containing this point
is denoted by $\C0$; it is a ray converging from $\alpha $ to, but disjoint from, the core of the inverse
limit space $\ILcores$.

Frequently, the Ingram Conjecture is posed for slopes $s, s' \in [\sqrt{2}, 2]$ only, because for
$0 < s \leq \sqrt{2}$, $\ILcores$ is decomposable. Since $\ILs$ is a single point for $s=0$ and a
single arc for $s \in (0,1]$, we will always assume that all slopes $s$ are greater than 1. The next two lemmas
show how to reduce the case $s \in (1,\sqrt{2}]$ to $s \in (\sqrt{2}, 2]$.

\begin{lem}\label{lem:renorm}
For $2^{1/2^{n+1}} \le s \le 2^{1/2^n}$, $n\in \N$, the core of the inverse limit space
$\ILcores$ is homeomorphic with two
copies of ${\underleftarrow\lim([0,1],T_{s^2})}$ joined at their endpoints.
\end{lem}

\begin{proof}
For this range of $s$, $T_s([c_2,p])=[p,c_1]$ and $T_s([p,c_1])=[c_2,p])$,
where $p:=\frac{s}{s+1}$ is the positive fixed point of $T_s$.
It follows that $\ILcores$ is homeomorphic with two copies of
${\underleftarrow\lim([p, c_1],T^2_{s})}$ joined at the endpoint
$(\ldots,p,p,p)$. Direct calculation
shows that, if $L$ is the orientation preserving affine homeomorphism from
$[p,c_1]$ onto
$[0,c_1(s^2)]$, then $L\circ T^2_s \circ L^{-1}=T_{s^2}$ on $[0,c_1(s^2)]$
and hence
${\underleftarrow\lim([p, c_1],T^2_{s})}$ is homeomorphic with  ${\underleftarrow\lim([0,1],T_{s^2})}$.
\end{proof}

\begin{lem}\label{lem:reduce}
Suppose that $2^{1/2^n}<  s\le 2^{1/2^{n-1}}$ and $2^{1/2^{n'}}<  s'\le 2^{1/2^{n'-1}}$, $n,n'\in \N$,
and suppose that $\ILs$ is homeomorphic with $\ILsp$. Then $n=n'$ and
assuming that the Ingram Conjecture holds for slopes $> \sqrt{2}$,
then also
${\underleftarrow\lim([0,1],T_{s^{2^{n-1}}})}$ is homeomorphic with
${\underleftarrow\lim([0,1],T_{(s')^{2^{n-1}}})}$.
\end{lem}

\begin{proof}
For $2^{1/2}<s<2$, $\ILs$ consists of a ray $\C0$ winding onto an indecomposable continuum,
namely $\ILcores$. It follows from Lemma~\ref{lem:renorm} that for $2^{1/2^n}<  s< 2^{1/2^{n-1}}$,
$\ILs$ consists of a ray winding onto a pair of rays, each winding onto a pair of rays,\ldots,
each winding onto a pair of rays, each of which winds onto an indecomposable continuum.
There are $2^{n-1}$ of these indecomposable continua, each homeomorphic with the core of the
inverse limit space  ${\underleftarrow\lim([0,1],T_{s^{2^{n-1}}})}$.
Hence if $\ILs$ is homeomorphic with $\ILsp$, then $n = n'$ and
${\underleftarrow\lim([0,1],T_{s^{2^{n-1}}})}$ is homeomorphic with
${\underleftarrow\lim([0,1],T_{(s')^{2^{n-1}}})}$.
To cover the remaining cases, note that if $s= 2^{1/2^{n-1}}$, then the only alteration needed in
the above description of $\ILs$ is that at the penultimate level, instead of a pair of rays winding
onto a pair of indecomposable subcontinua, we just have two indecomposable subcontinua (each homeomorphic
with ${\underleftarrow\lim([0,1],T_2)}$) joined at their common endpoint.
It is clear in this case that if $\ILsp$ is homeomorphic with $\ILs$, then $s'=s$.
\end{proof}
\begin{df}\label{df:dbar} The {\em arc-length} or {$\bar d$ metric} on $\C0$ is defined as
\[
\bar d(x,y) = s^p |x_{-p} - y_{-p}|
\]
for each $p$ so that $\pi_p:[x,y] \to [0,1]$ is injective.
\end{df}
If $x, y \in \C0$, then we denote by
$[x, y]$ the arc between $x$ and $y$, and by $(x, y)$ the interior of
the arc $[x, y]$. We write $x \preceq y$ if $x \in [\alpha, y]$,
\ie $\bar d(\alpha,x) \leq \bar d(\alpha, y)$.
\begin{df}\label{df:chain}
A continuum is {\em chainable} if for every $\eps > 0$, there is a cover
$\{ \ell^1, \dots , \ell^n\}$ of open sets (called {\em links}) of diameter
$< \eps$
such that $\ell^i \cap \ell^j \neq \emptyset$ if and only if $|i-j| \leq 1$.
Such a cover is called a {\em chain}.
Clearly the interval $[0,s/2]$ is chainable.
We call $\chain_p$ a {\em natural} chain of $\ILs$ if
\begin{enumerate}
\item there is a chain $\{ I^1_p, I^2_p, \dots , I^n_p \}$ of $[0,s/2]$,
with the relatively open interval $I^j_p$ and $I^{j+1}_p$
adjacent for all $1 \leq j < n-1$, such that
$\ell^j_p := \pi_p^{-1}(I^j_p)$ are the links of $\chain_p$;
\item each point $x \in \cup_{i=0}^p T_s^{-i}(c)$ is the boundary
point of some link $I^j_p$;
\item for each $i$ there is $j$ such that $T_s(I^i_{p+1}) \subset I^j_p$.
\end{enumerate}
Let us define $\width(\chain_p) := \max_j |I^j_p|$.
If $\width(\chain_p) < \eps s^{-p}/2$
then $\mesh(\chain_p) := \max\{ \diam(\ell) :\ell \in \chain_p\} < \eps$, which
shows that $\ILs$ is indeed chainable.

Condition 3.\ ensures that $\chain_{p+1}$ \emph{refines}  $\chain_p$
(written $\chain_{p+1} \preceq \chain_p$).
\end{df}
\begin{df}\label{df:p-point}
A point $x = (\dots , x_{-2}, x_{-1}, x_0) \in \C0$ is called a {\em $p$-point} if $x_{-j} = c$ for some
$j \geq p$. For the largest such $j$, the number $L_p(x) := j-p$ is called the {\em $p$-level}. In
particular, $x_0 = T_s^{p+L_p(x)}(c)$. The ordered set of all $p$-points of composant $\C0$ is denoted by
$E_p$, and the ordered set of all $p$-points of $p$-level $l$ by $E_{p,l}$. Given an arc $A \subset \C0$
with successive $p$-points $x^0, \dots , x^n$, the {\em $p$-folding pattern} of $A$, denoted by $FP_p(A)$,
is the sequence $FP_p(A) = L_p(x^0), \dots , L_p(x^n)$. The {\em folding pattern of composant} $\C0$,
denoted by $FP(\C0)$, is the sequence $L_p(z^1), L_p(z^2), \dots , L_p(z^n), \dots$, where
$E_p = \{ z^1, z^2, \dots , z^n, \dots \}$ and $p$ is any nonnegative integer. Let $q \in \N$, $q > p$, and
$E_q = \{ y^0, y^1, y^2, \dots \}$. Since $\sigma^{q-p}$ is an order-preserving homeomorphism of $\C0$,
it is easy to see that, for every $i \in \N$, $\sigma^{q-p}(z^i) = y^i$ and $L_p(z^i) = L_q(y^i)$. Therefore,
the folding pattern of $\C0$ does not depend on $p$.
\end{df}
For the above arc $A$, the projection $\pi_p:A \to [0,s/2]$ need not be
injective; so the folding pattern of $A$ can be very long and $A$ may pass
through the same link $\ell^j$ of the natural chain $\chain_p$ many times.
If $A^j$ is an arc component of $A \cap \ell^j$, then we say that
$A^j$ {\em goes straight} through $\ell^j$ if $\pi_p|_{A^j}$ is injective;
otherwise it {\em turns} in $\ell^j$.
If $A^j$ turns in $\ell^j$, then $A^j$ contains at least one $p$-point.

\begin{df}\label{df:linksym}
Let $\ell^0, \ell^1, \dots, \ell^k$ be those links in $\chain_p$ that are successively visited by an arc
$A = [u,v] \subset \C0$ (hence $\ell^i \neq \ell^{i+1}$, $\ell^i \cap \ell^{i+1} \neq \emptyset$ and
$\ell^i = \ell^{i+2}$ is possible if $A$ turns in $\ell^{i+1}$). Let $A^i \subset \ell^i$ be the
corresponding arc components such that $\Cl A^i$ are subarcs of $A$.
We call the arc $A$
\begin{itemize}
\item {\em $p$-link-symmetric} if $\ell^i = \ell^{k-i}$ for $i = 0, \dots, k$;
\item {\em maximal $p$-link-symmetric} if it is  $p$-link-symmetric
and there is no $p$-link-symmetric arc $B \supset A$ and passing through more links than $A$;
\item \emph{$p$-symmetric} if $\pi_p(u) =  \pi_p(v)$ and if for $A \cap E_p = \{ x^0, \dots , x^n \}$
we have $L_p(x^i) = L_p(x^{n - i})$ for every $i = 0, \dots , n$.
\end{itemize}
In any of these cases, the $p$-point of $A^{k/2}$ with the highest $p$-level is called the {\em center} of $A$, and the link $\ell^{k/2}$ is called the {\em central link} of $A$.
\end{df}

It is easy to see that if $A$ is $p$-symmetric, then $n$ is even and
$L_p(x^{n/2}) = \max \{ L_p(x^i) : x^i \in A \cap E_p \}$. Clearly, every $p$-symmetric arc is
$p$-link-symmetric as well, but the converse does not hold.

\begin{df}\label{df:snappy}
Let $(s_i)_{i \in \N}$ be a sequence of $p$-points such that
$0 \leq L_p(x) < L_p(s_i)$ for every $p$-point $x \in (\alpha , s_i)$. We call
$p$-points satisfying this property \emph{snappy}.
\end{df}
Since for every slope $s > 1$ and $p \in \N_0$, the sequence $FP(\C0)$
starts as $0 \ 1 \ 0 \ 2 \ 0 \ 1 \ \dots$, and since by definition $L_p(s_1) > 0$, we have
$L_p(s_1) = 1$. Also, since $s_i = \sigma^{i-1}(s_1)$, $L_p(s_i) = i$, for every $i \in \N$. Note that
the snappy $p$-points depend on $p$: if $p \geq q$, then the snappy $p$-point $s_i$ equals the snappy
$q$-point $s_{i+p-q}$.

Let us extend the notion of folding pattern as follows.
A sequence $e_1, \dots , e_k$ is the {\em folding pattern} of $T^j|_H$ for an interval $H \subset [0, 1]$
if $c_{e_1} = T^j(x_1), \dots , c_{e_k} = T^j(x_k)$, where $x_1 < \dots < x_k$ are the critical
points of $T^j$ on $H$. (If $0 \in H$, then the folding pattern starts with $*$ by convention, just as
$*$ denotes the conventional $p$-level of $\alpha$.)
In this extended terminology, the $p$-folding pattern of $[\alpha, s_{j+1}]$ is the same as
the folding pattern of $T^j$ on $[0, c_1]$, independently of $p$.

Measured in arc-length,
$\bar d(\alpha ,s_1) = \frac12 s^p$, and since $\sigma(s_i) = s_{i+1}$ we
obtain
\begin{equation}\label{eq:dbarmi}
\bar d(\alpha , s_i) = \frac12s^{p-1} s^i \quad \text{ for all } i \geq 1.
\end{equation}

\section{Maximal Link-Symmetric Arcs}\label{sec:lemmas}

In this section we establish upper bounds for the lengths of $p$-link-symmetric arcs.
The Ingram Conjecture was previously proved for all tent-maps with a (pre)periodic critical point,
see \cite{Stim}. So let as assume
from now on that the slope $s$ is such that  $c$ is not (pre)periodic. Throughout
this section we use the notation $T := T_s$, $a_k := T^k(a)$ for any point or interval
(except for the precritical points $z_k$ in Definition~\ref{def:ccp} below), and $\hat a := 1-a$
is the symmetric point around $c$.

\begin{df}\label{def:eps-symmetry}
Given $\eps > 0$ and $H := [a,b] \subset [0,c_1]$,
we say that $T^n|_H$ is \emph{$\eps$-symmetric},
if $|T^n(a+t) - T^n(b-t)| < \eps$ for all $0 \leq t \leq b-a$.
\end{df}
If $\width(\chain_p) < \eps$ and the arc $J \subset [\alpha , s_k]$
is $p$-link-symmetric, then $\pi_{p+k} :J \to H := \pi_{p+k}(J)$ is one-to-one and $T^k|_H$ is
$\eps$-symmetric.
\begin{df}\label{def:eps-periodic}
We say that $T^n|_H$ is \emph{$\eps$-periodic} of period $2\eta$ if
$|T^n(t) - T^n(t+2\eta)| < \eps$ for all $t, t+2\eta \in H$.
\end{df}
If $T^n|_H$ is $\eps$-symmetric around two centers that are $\eta$
apart, then $T^n|_H$ is $\eps$-periodic with period $2\eta$.
We will explain this fact in more detail in the proof
of Proposition~\ref{prop:max_link_symmetry}, where it is used several times.

\begin{df}\label{def:ccp}
We call $z_k$ a \emph{closest precritical point} if $T^k(z_k) = c$ and  $T^k$ maps
$[c,z_k]$ monotonically onto $[c_k, c]$. Clearly, if $z_k$ is a closest precritical points,
so is $\hat z_k$.
\end{df}

\begin{lem}\label{lem:nbh}
There are infinitely many $N$ and closest precritical points $z_N$ such that
$\theta_N := \min\{ |c_i-c| : 0 < i \leq N \} > |z_N - c|$.
\end{lem}

\begin{proof}
If $c$ is not recurrent, then $\theta_n \not\to 0$ and the lemma is trivial. So let us assume that
$c$ is recurrent, but obviously not periodic. Let $n$ be such that $|c_n-c| = \theta_n$.

If $x \mapsto |T^n(x)-c|$ has a local maximum at $c$, then $T^n([c,c_n]) \owns c$. Indeed, if this
were not the case, then by the choice of $n$, $T^n$ maps $[c,c_n]$ in a monotone
fashion into $[c,c_n]$, which is clearly impossible for tent maps with slope $> 1$. So in this case,
$z_n \in [c_n, \hat c_n]$ and the lemma holds with $N = n$.

So assume now that $x \mapsto |T^n(x)-c|$ has a local minimum at $c$. Take $m \in \N$ minimal such that
the closest precritical $z_m \in [\hat c_n,c_n]$. We will show that $c_j \notin [\hat z_m , z_m]$ for
$n < j \leq m$. If $j = m$, then $x \mapsto |T^j(x)-c|$ has a local
maximum at $c$, and we can argue as above.
So assume by contradiction that $c_j \in [\hat z_m , z_m]$ for some $n < j < m$. If $x \mapsto |T^j(x)-c|$
has a local maximum at $c$, then the closest precritical point $z_j$ satisfies
$T^j([c,z_m]) \subset T^j([c, z_j]) = [c_j,c] \subset [\hat z_m, c]$ or $[c, z_m]$. This implies that either
$[c,z_m]$ or $[\hat z_m,c]$ is mapped monotonically into itself by $T^j$, which is impossible. The remaining
possibility is that $x \mapsto |T^j(x)-c|$ has a local minimum at $c$. In this case, $T^{j-n}$ maps
$[z_m, c_n]$ monotonically onto $[w, c_j]$. If $c \in (w, c_j)$, then $m \in \N$ cannot be minimal such that
$T^m([c,c_n]) \owns c$. If $c \notin (w, c_j)$, then $w \in [\hat c_n, c_n] \cap T^{(j-n)-m}(c)$, and since
$-m < (j-n)-m < 0$, $m$ is again not minimal such that $T^m([c,c_n]) \owns c$.

Take $N = m$ and the lemma follows.
\end{proof}

Take $N_0$ as in Lemma~\ref{lem:nbh} and so large that $s^{N_0}
> 100$. Let $N \geq N_0$ from Lemma~\ref{lem:nbh} be so large that
\begin{equation}\label{eq:delta0}
\delta := |z_N - c| < |z_{N_0} - c|/100.
\end{equation}
Then $|c_n-c| \geq s^n |z_n-c| \geq s^{N_0}|z_N-c| > 100 \delta$
for every $N_0 \leq n \leq N$ by the choice of $N_0$ and
$|c_n-c| > |z_{N_0}-c| > 100\delta$ for $n \leq N_0$ by the choice of $N$.

\begin{lem}\label{lem:Jdelta}
Given $\delta$ as in \eqref{eq:delta0}, there exists $r_0 = r_0(\delta)$ such that for every interval
$\tilde J$ with $|\tilde J| \geq 22\delta$, there exist $l \leq r_0N$ and an interval $J$ with
$|J| \geq 18\delta$ and concentric with $\tilde J$, such that $T^l_s|_J$ is  monotone and
$J_l := T^l(J) \supset [c-\delta, c+\delta]$.
\end{lem}
\begin{proof}
Let $x$ be the center of $\tilde J$ and take $m \geq 0$
minimal such that $\tilde J_m \owns c$; hence $T^m|_{\tilde J}$ is
monotone.

Clearly, $m \leq (r_0-1)N$ for some $r_0 \geq 1$ depending only on
$\delta$. If $\bd \tilde J_m$ is $\delta$-close to $c$, then we take
$J' \subset \tilde J$ centered at $x$ and slightly smaller such that
$c \in \partial J'_m$ and $m' > m$ minimal such that $J'_{m'}$
contains $c$ in its interior. Since $|J'_m| > 20\delta$, it contains
$z_N$ or $\hat z_N$ as in \eqref{eq:delta0}, and $m'-m \leq N$
and $|c_{m'-m}-c| \geq \delta$ by Lemma~\ref{lem:nbh}.

If at iterate $m'$ the other boundary point of $J'$ is $\delta$-close to $c$,
then $m'-m < N$. We take the interval $J'' \subset J'$ centered at $x$ slightly
smaller such that $c \in T^{m'}(\bd J'')$ and take $m'' > m'$ minimal such that
$c$ is an interior point of $T^{m''}(J'')$.
Since $T^{m'}(z_N) \in T^{m'}(J'')$, and by \eqref{eq:delta0} again, $m \leq
m' \leq m'' \leq m+N$ and $\bd J''_{m''}$ is not $\delta$-close to
$c$. In each case, there is $l \leq r_0 N$ and $J \in \{ \tilde J,
J', J''\}$ so that the lemma holds.
\end{proof}

For interval $H =: [a,b]$ with center $x$ we formulate the following property:
\begin{equation}\label{eq:H}
c \in H \text{ and } \delta < \min \{ |c-a|,  |c-b|, |c-x|\}.
\end{equation}

\begin{pro}\label{prop:max_link_symmetry}
Assume that $s \in [1, 2]$ is such that $c$ is not (pre)periodic.
There exists $\eps > 0$ such that if $H$ satisfies \eqref{eq:H},
then $T^n|_H$ is not $\eps$-symmetric for any $n \in \N_0$.
\end{pro}

\begin{proof}
We will prove Proposition~\ref{prop:max_link_symmetry} using the induction hypothesis:
\begin{equation}
\text{if $H$ satisfies \eqref{eq:H}, then $T^n|_H$ is not
$\eps$-symmetric. }\tag{IH$_n$}
\end{equation}
Take $N_0$, $N$ and $\delta$ as in \eqref{eq:delta0},
$r_0$ as in Lemma~\ref{lem:Jdelta} and $H$ that satisfies \eqref{eq:H}.

Let $\eps \in (0,\delta)$ be so small that
\begin{equation}\label{eq:eps1}
\eps < \min\{ |c_i - c_j| : 0 \leq i < j \leq (2 + r_0)N  \}.
\end{equation}

Since $c$ lies off-center in $H$ by at least $\delta$, by the choice of $\eps$, (IH$_k$) holds for all
$k \leq (2+r_0)N$.
Assume now that (IH$_j$) holds for all $j < n$. We will prove (IH$_n$), but first, continuing with
the interval $\tilde J$ of Lemma~\ref{lem:Jdelta}, we prove the following lemma.

\begin{lem}\label{lem:eps2}
Let $\tilde J$ be an interval of length $|\tilde J| \geq 22\delta$ centered at $c_k$ for some
$1 \leq k \leq 2N$. If $T^j|_{\tilde J}$ is $\eps$-symmetric for some $0 \leq j \leq n$, then
the interval $J_l := T^l(J)$ from Lemma~\ref{lem:Jdelta} satisfies condition \eqref{eq:H}.
\end{lem}
\begin{proof}
We know already from Lemma~\ref{lem:Jdelta} that $J_l \supset
[c-\delta, c+\delta]$. Hence if \eqref{eq:H} fails, then $\eta :=
|c_{k+l}-c| \leq \delta$. Since $T^l|_J$ is monotone, $j > l$.
Therefore $T^{j-l}|_{J_l}$ is $\eps$-symmetric around $c_{k+l}$ and
symmetric around $c$, and it follows that $T^{j-l}|_{J_l}$ is
$\eps$-periodic with period $2\eta$. Indeed, by symmetry around $c$,
$T^{j-l}_{J_l}$ is $\eps$-symmetric around the symmetric point $\hat
c_{k+l}$. Hence $T^{j-l}_{J_l}$ must also be $\eps$-symmetric around
the points $c\pm 2\eta$, which are the reflections of $c$ in
$ c_{k+l}$ and $\hat c_{k+l}$, etc. Extending these symmetries, we see that
$|T^{j-l}(t) - T^{j-l}(t+2\eta)| < \eps$ for all $t, t+2\eta \in J_l$,
so $T^{j-l}|_{J_l}$ is $\eps$-periodic with period $2\eta$.
Even more, $T^{j-l}|_{J_l}$ is $\eps$-symmetric around $c+2i\eta$ on every separate subarc
$P_i := [c+(2i-1)\eta, c+(2i+1)\eta] \subset J_l$.

Recall that $1\leq k \leq 2N$ and $l \leq r_0N$, so
we have $\eta > \eps$ by the choice of $\eps$ in \eqref{eq:eps1}.
Since $|J_l| \geq 18\delta = 18|z_N-c|$, one of the components
of $J_l\setminus\{ c \}$, say the one containing $z_N > c$, has length $\geq 9\delta$.
We can take $r \leq N$ minimal such that $z_r \in [c+\delta,c+8.9\delta]$.
Take $i \in \Z$ such that if
\begin{equation}\label{eq:choose_i}
z_r \in \left\{ \begin{array}{ll}
(c,c+ 4.3\delta ], & \text{ then }
 c+2i\eta \in (z_r+0.1\delta, z_r+2.1\delta), \\[2mm]
(c+4.3\delta, c+8.9\delta ], & \text{ then } c+2i\eta \in
(z_r-2.1\delta, z_r-0.1\delta).
\end{array} \right.
\end{equation}
Let $H \subset J_l$ be the longest interval centered at $x :=
c+2i\eta$ on which $T^r|_H$ is monotone. Then $H \owns z_r$, and
$T^{j-l}|_H$ and $T^{j-l-r}|_{H_r}$ are $\eps$-symmetric. We
will show that $H_r$ satisfies $\eqref{eq:H}$. Indeed,
since $|z_r-c| \leq 9\delta < |z_{N_0} - c|/10$ (so $r > N_0$)
by \eqref{eq:delta0} and $|x-z_r| \geq \delta/10$,
we have $|x_r - c| = s^r|x-z_r| \geq 2^{N_0/2}\delta/10 > \delta$.
If $|z_r - \bd H| \geq \delta/10$, then
$|c-\bd H_r|  > \delta$ for the same reason.
If on the other hand there is a point $y \in \bd H$ such that
$|y-z_r| < \delta/10$, then $y$ has to be a precritical point.
By the choice of $r$, $y = z_{r'} \in
(c+8.9\delta, c+9\delta]$ for some $r' < r$.
By the choice of $N$ and Lemma~\ref{lem:nbh},
$|y_r-c| = |c_{r-r'}-c| \geq \delta$.

This shows that $H_r$ satisfies $\eqref{eq:H}$, but also
$T^{j-l-r}|_{H_r}$ is $\eps$-symmetric around $x_r$, and this
contradicts (IH$_{j-l-r}$), proving this lemma.
\end{proof}

Combining the induction hypothesis (IH$_n$) and Lemma~\ref{lem:eps2}, we have proved the
following stronger property.

\begin{cor}\label{cor:eps3}
If $\tilde J$ is centered at $c_k$ for some $1 \leq k \leq 2N$ and $|\tilde J| \geq 22\delta$,
then $T^j|_{\tilde J}$ is not $\eps$-symmetric for $j \leq n$.
\end{cor}

Now we continue the induction on $n$ and assume by contradiction
that $T^n|_H$ is $\eps$-symmetric for some $H$ satisfying
\eqref{eq:H} and for $\eps$ satisfying \eqref{eq:eps1}.
Let $[a',b'] := H' \subset H$ be centered around $x$ such that $c \in
\partial H'$. Assume without loss of generality that $c =  a'$ is
the left endpoint of $H'$, and let $L$ and $R$ be intervals of
length $\delta$ at the left and right side adjacent to $H'$. Since
$|H'| \geq \delta$, so $H' \owns z_N$ or $\hat z_N$, there is $0 < k
\leq N$ minimal such that $c \in H'_k$.
Clearly $|H'_k| > |L_k| = |R_k| \geq 100\delta$.
We distinguish four cases:\\[1mm]
{\bf Case I:}
$H'_k$ satisfies \eqref{eq:H}. Then by (IH$_{n-k}$),
$T^{n-k}|_{H'_k}$ cannot be $\eps$-symmetric,
and neither can $T^n|_{H'}$ or $T^n|_H$.\\[1mm]
\begin{figure}[ht]
\unitlength=7mm
\begin{picture}(20,11.3)(0,-0.5)
%
%
\put(0.5,10){$\overbrace{\qquad\qquad\qquad\qquad\qquad\qquad\qquad}$}
\put(4.4,10.55){$H$}
\thicklines
\put(1, 9){\line(1,0){1.5}}\put(1, 9.03){\line(1,0){1.5}}\put(1.5,9.2){$L$}
\put(4.3,9.3){$H'$}
\put(6.5, 9){\line(1,0){1.5}}\put(6.5, 9.03){\line(1,0){1.5}}\put(7,9.2){$R$}
\thinlines
\put(0.5, 9){\line(1,0){8.3}}
\put(2.5, 8.9){\line(0,1){0.2}}\put(1.7,8.4){$a'=c$}
\put(4.5, 8.9){\line(0,1){0.2}}\put(4.3,8.4){$x$}
\put(6.5, 8.9){\line(0,1){0.2}}\put(6.3,8.4){$b'$}
\put(1,7.5){\vector(0,-1){1}}\put(0,7){$T^k$}
\thicklines
\put(2.5, 6){\line(3,-1){1.5}}\put(2.5, 6.03){\line(3,-1){1.5}}\put(3,5.1){$L_k$}
\put(4.3,6.3){$H'_k$}
\put(6.5, 6){\line(1,0){1.5}}\put(6.5, 6.03){\line(1,0){1.5}}\put(7,6.2){$R_k$}
\thinlines
\put(2.5, 6){\line(1,0){4}}
\put(2.5, 5.9){\line(0,1){0.2}}\put(2.3,5.4){$c_k$}
\put(4.5, 5.9){\line(0,1){0.2}}\put(4.3,5.4){$x_k$}
\put(6.5, 5.9){\line(0,1){0.2}}\put(6.3,5.4){$b'_k$}
\put(1,4.5){\vector(0,-1){1}}\put(0,4){$T$}
\thicklines \put(0.5, 3.2){\line(5,-1){2}}\put(0.5,
3.23){\line(5,-1){2}}\put(1.2,3.5){$R_{k+1}$} \put(2.5,
0.8){\line(5,-1){2}}\put(2.5,
0.83){\line(5,-1){2}}\put(3.4,-0.2){$L_{k+1}$} \thinlines
\put(1.5,3){\line(5,-1){6}} \put(7.5, 1.8){\line(-5,-1){5}}
\put(2.5, 0.7){\line(0,1){0.2}}\put(1.8,0.3){$c_{k+1}$} \put(2.5,
0.7){\line(0,1){0.2}}\put(1.8,0.3){$c_{k+1}$} \put(6.8,
1.8){\line(0,1){0.2}}\put(6.3,2.4){$x_{k+1}$} \put(2.5,
2.6){\line(0,1){0.2}}\put(2.5,3){$c_{k+1} \approx b'_{k+1}$}
\put(7.6,2){$c_1$}
%
%
\put(10.5,10){$\overbrace{\qquad\qquad\qquad\qquad\qquad\qquad\qquad}$}
\put(14.4,10.55){$H$}
\thicklines
\put(11, 9){\line(1,0){1.5}}\put(11, 9.03){\line(1,0){1.5}}\put(11.5,9.2){$L$}
\put(14.3,9.3){$H'$}
\put(16.5, 9){\line(1,0){1.5}}\put(16.5, 9.03){\line(1,0){1.5}}\put(17,9.2){$R'$}
\thinlines
\put(10.5, 9){\line(1,0){8.3}}
\put(12.5, 8.9){\line(0,1){0.2}}\put(11.7,8.4){$a'=c$}
\put(14.5, 8.9){\line(0,1){0.2}}\put(14.3,8.4){$x$}
\put(16.5, 8.9){\line(0,1){0.2}}\put(16.3,8.4){$b'$}
\put(11.3,7.6){\vector(0,-1){1}}\put(9.9,7){$T^{k+1}$}
\thicklines
\put(12.5, 6){\line(3,-1){1.5}}\put(12.5, 6.03){\line(3,-1){1.5}}\put(12.8,5.1){$L_{k+1}$}
\put(13.5,6.4){$H'_{k+1}$}
\put(16.5, 6){\line(-3,1){1.5}}\put(16.5, 6.03){\line(-3,1){1.5}}\put(15.2,6.7){$R'_{k+1}$}
\thinlines
\put(12.5, 6){\line(1,0){4}}
\put(12.5, 5.9){\line(0,1){0.2}}\put(11.7,5.6){$c_{k+1}$}
\put(14.5, 5.9){\line(0,1){0.2}}\put(14.3,5.4){$x_{k+1}$}
\put(16.5, 5.9){\line(0,1){0.2}}\put(16.7,5.6){$c_1 \approx b'_{k+1}$}
\put(11.3,4.6){\vector(0,-1){1}}\put(9.9,4){$T^{j+1}$}
\thicklines
\put(12.5, 0.8){\line(5,-1){4}}\put(12.5, 0.83){\line(5,-1){4}}\put(13.3,-0.3){$L_{k+j+2}$}
\put(11.5, 3){\line(5,1){4}}\put(11.5, 3.03){\line(5,1){4}}\put(12.1,3.9){$R'_{k+j+2}$}
\thinlines
\put(11.5, 3){\line(5,-1){6}}
\put(17.5, 1.8){\line(-5,-1){5}}
\put(12.5, 0.7){\line(0,1){0.2}}\put(11,0.3){$c_{k+j+2}$}
\put(10.2,2.8){$c_{j+2}$}
\put(17.7,2){$c_1 \approx x_{k+j+2}$}
\put(11.45,2.5){$\underbrace{\quad}_{\text{\small \it hook}}$}
\put(11.5,2.7){\dashbox{0.2}(1,0.7)}
\end{picture}
\caption{An illustration of Cases II (left) and IV (right).}
\label{fig:CaseII+IV}
\end{figure}
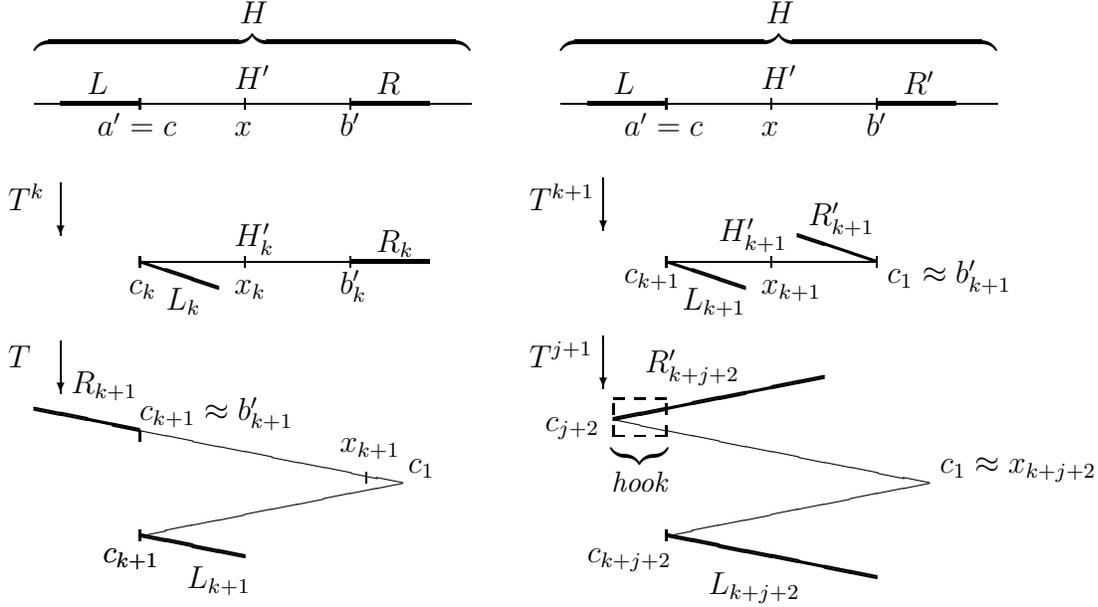
\noindent
{\bf Case II:} $|x_k-c| < \delta$, see Figure~\ref{fig:CaseII+IV} (left).
If the length of the interval $T^{n-k}([x_k,c])$ exceeds $\eps$,
then since $T^{n-k}$ is also symmetric around $c$, $T^{n-k}$
must be $\eps$-symmetric on $H'_k$ both with center $x_k$
and with center $\hat x_k$,
and therefore $\eps$-periodic on $H'$ with period $2\eta := 2|x_k-c|$.
We use the same argument as in the proof of Lemma~\ref{lem:eps2}:
$T^{n-k}$ is $\eps$-symmetric on each interval
$P_i := [c+(2i-1)\eta, c+(2i+1)\eta]$ for each $i \in \Z$ such that
$P_i \subset H'_k$. Since $|H'_k| \geq 100\delta \geq 100 \eta$,
$P_i \subset H'_k$ for at least $-25 \leq i \leq 25$.
Take $r \leq N$ minimal such that $[z_r-\delta/10, z_r+\delta/10] \subset  H'_k$,
and $i \in \Z$ as in \eqref{eq:choose_i}, and
$H'' \subset H'$ such that $H''_k$ is the maximal interval centered
at $c+2i\eta$ on which $T^r$ is monotone.
As before, $T^{n-(k+r)}|_{H''_{k+r}}$ is the $\eps$-symmetric but
$H''_{k+r}$ satisfies \eqref{eq:H}.
This would contradict (IH$_{n-(k+r)}$), so it cannot occur.

If on the other hand the length of $T^{n-k}([x_k,c])$ is less then
$\eps$, then we might as well have chosen $x$ such that $x_k =
c$. This means that the intervals $L_{k+1}$ and $R_{k+1}$  are
adjacent, see Figure~\ref{fig:CaseII+IV} (bottom left). More
precisely, they are adjacent except for an error which does not show
at $\eps$-scale under the iterate $T^{n-(k+1)}$, so by a negligible
adjustment, we can assume that they form an interval of length $\geq
100\delta$ with center $c_{k+1}$. Since $k+1 \leq 2N$, Corollary~\ref{cor:eps3}
implies that $T^{n-(k+1)}|_{L_{k+1} \cup R_{k+1}}$
and hence $T^n|_H$ are not $\eps$-symmetric.
\\[1mm]
{\bf Case III:} $|a'_k-c| < \delta$. Since $k \leq N$, the choice of
$\delta$ renders this impossible.\\[1mm]
{\bf Case IV:} $|b'_k-c| < \delta$, see Figure~\ref{fig:CaseII+IV}
(right). Replace $R$ by the largest interval $R' \subset H \cup R$
with $R'\cap R \neq \emptyset$ such that $c \in \bd R'_k$ and
$T^k|_{R'}$ is monotone. If $c \in \bd R'_l$ for some $0 \leq l <
k$, then $R'_k = [c,c_{k-l}]$, so $|R'_k| \geq \delta$ by
Lemma~\ref{lem:nbh}. Also rename $H' \setminus R'$ to $H'$. Hence
$T^{k+1}|_{L \cup H' \cup R'}$ has three branches, $s \delta \leq
|R'_{k+1}|$ and $100\delta \leq |L_{k+1}| \leq |H'_{k+1}|$.

Let $j > 0$ be minimal such that $T^{k+j+1}(H') \owns c$. If
$H'_{k+j+1} = [c_{k+j+1}, c_{j+1}]$, which is centered at
$x_{k+j+1}$, satisfies \eqref{eq:H}, then we can invoke
(IH$_{n-(k+j+1)}$), so assume that this is not the case. Since $|L|
\geq \delta$, so $L \owns z_N$ or $\hat z_N$, we have $j \leq k+j+1
\leq N$. Therefore both $|c_{j+1}-c| > \delta$ and $|c_{k+j+1}-c|
\geq \delta$.

Thus if \eqref{eq:H} fails, we must have $|x_{k+j+1} - c| < \delta$.
If in the remaining $n-(k+j+1)$ iterates, the arc $[x_{k+j+1} , c]$
grows to length $> \eps$,
then, as in Case II, $T^n|_{H'}$ must contain a large $\eps$-periodic
arc, to which we apply the same argument as in Case II (\ie the argument of Lemma~\ref{lem:eps2}).
The remaining possibility is that $x_{k+j+1}$ is so close to $c$ that on
an $\eps$-scale, we may as well assume that $x_{n+k+1} = c$.

Both $c_{k+j+2} = a'_{k+j+2}$ and $c_{j+2} \approx b'_{k+j+2}$
are local minima of $T^{k+j+2}|_{L \cup H' \cup R'}$,
see Figure~\ref{fig:CaseII+IV} (bottom right).
Assume without loss of generality that
$c_{j+2} < c_{k+j+2}$, so  $R'_{k+j+2}$ has a small extra hook before
joining up with $L_{k+j+2}$.
As we assumed that $T^n|_H$ is $\eps$-symmetric, the effect of this hook
needs to be `$\eps$-repeated' near $a'$ in $L$.
But $L_{k+j+2}$ and $R'_{k+j+2}$ overlap, so
in $R'$, the same effect needs to be $\eps$-repeated next to the first
hook. Continuing this way, we find that $T^{n-(k+j+2)}$ is
$\eps$-periodic over the entire length of $R'_{k+j+2}$.

Take $i$ minimal such that $R'' := T^i(R'_{k+j+2}) \owns c$.
Since $|R'_k| \geq \delta$ we have $j+i+2 < N$, $|R''| \geq 100\delta$ and
$|\bd R'' - c| \geq \delta$.
Therefore $T^{n-(k+j+i+2)}|_{R''}$
is $\eps$-periodic of period $2\eta$, where
the length of the hook after $i$ more iterates is
$\eta := |c_{j+i+2} - c_{k+j+i+2}| > \eps$, because $k+j+i+2 \leq 2N$
and by the choice of $\eps$ in \eqref{eq:eps1}.
If $\eta < 10 \delta < |R''|/10$, then $T^{n-(k+j+i+2)}|_{R''}$
is $\eps$-periodic with at least $5$ adjacent intervals $P$ of
length $2\eta$ around the center of which $T^{n-(k+j+i+2)}|_{R''}$ is
$\eps$-symmetric.
So we can find a new interval $H'' \subset R''$ centered around the center of one of these $P$s such that $H''$ satisfies \eqref{eq:H}.
But this contradicts (IH$_{n-(k+j+i+2)}$).

If $\eta \geq 10\delta$, then we let $H''$ be the arc of length
$22\delta$ centered at $c_{k+j+i+2}$. Again, since $k+j+i+2 \leq
2N$, the iterate $T^{n-(k+j+i+2)}$ cannot be $\eps$-symmetric
on $H''$ by Corollary~\ref{cor:eps3}. But then the assumed
$\eps$-symmetry of $T^n|_H$ does not extend beyond $H'$, and
Case IV follows.
\\[1mm]
This proves the inductive step and hence the proposition.
\end{proof}

Let $\kappa := \min\{ i \geq 3 : c_i \geq c\}$. Then $\kappa < \infty$ provided $1 < s < 2$.
Let $\cdots < c_{-3} < c_{-2} < c_{-1} < c_0 = c$ be the successive
precritical points on the left of $c$ with $T^j(c_{-j}) = c$.
Since $c_{\kappa-1} < c < c_\kappa$, we have $c_{2-\kappa} < c_2 < c_{3-\kappa}$.
Let $\delta = |z_N - c|$ as in \eqref{eq:delta0} be so small (\ie $N$ as in
Lemma~\ref{lem:nbh} so large) that
\begin{equation}\label{eq:delta}
\delta  <\frac1{30}  \min\{ |c_{-1} - c_{-2}|, |c_{-1} - \hat c_1|, |c_2 - c_{2-\kappa}| \},
\end{equation}
where $\hat c_1 = 1-c_1 = 1-s/2$. Assume that $s \in [1,2]$ is such
that $c$ is not (pre)periodic, and take $\eps$ is as in
\eqref{eq:eps1} in the proof of
Proposition~\ref{prop:max_link_symmetry}.

Let $(A_i)_{i \in \N}$ be the sequence of maximal $p$-link-symmetric arcs with center $s_i$ for every
$i \in \N$. Recall that $(s_i)_{i \in \N}$ is the sequence of snappy $p$-points
(see Definition~\ref{df:snappy}) and that $\width(\chain_p) := \max_j |I^j_p|$.
\begin{lem}\label{lem:snappy}
If $\width(\chain_p) < \eps$, then $A_i$ contains exactly $\kappa$ snappy
$p$-points for each $i \geq \kappa-1$, namely $s_{i-\kappa+2}, s_{i-\kappa+3}, \dots, s_{i+1}$, and
$s_{i-\kappa+2}$ is an interior point of $A_i$.
\end{lem}
\begin{proof}
Let $H$ be the interval centered at $c_2$ such that $c$ is the left endpoint of
$H_{\kappa-2} := T^{\kappa-2}(H)$. Then $|H| \geq 22\delta$ by the choice of $\delta$, so by
Proposition~\ref{prop:max_link_symmetry} and Lemma~\ref{lem:eps2} in particular, $T^{p+i-1}|_H$
cannot be $\eps$-symmetric.

\begin{figure}[ht]
\unitlength=8mm
\begin{picture}(16,6)(-1,-1)
\thicklines
\put(3, 4){\line(1,0){10}}\put(3,3.98){\line(1,0){10}}
\put(3, 3.96){\line(1,0){10}}\put(8.3,4.3){$J$}
\put(3, 2){\line(1,0){10}}\put(3,1.98){\line(1,0){10}}
\put(3, 1.96){\line(1,0){10}}
\put(4.5, 0){\line(1,0){8.5}}\put(4.5,-0.02){\line(1,0){8.5}}
\put(4.5, -0.04){\line(1,0){8.5}}
\put(13,-0.145){\oval(0.22,0.22)[r]}
\put(13,-0.145){\oval(0.26,0.26)[r]}
\put(13, -0.14){\oval(0.29,0.29)[r]}
\put(4.5, -0.025){\line(1,0){8.5}}\put(4.5,-0.26){\line(1,0){8.5}}
\put(4.5, -0.29){\line(1,0){8.5}}
\put(13,1.875){\oval(0.26,0.26)[rt]}
\put(13,1.86){\oval(0.25,0.25)[rt]}
\put(13,1.84){\oval(0.24,0.24)[rt]}
\put(4.5,-0.38){\oval(0.25,0.25)[lt]}
\put(4.5,-0.39){\oval(0.24,0.24)[lt]}
\put(4.5,-0.4){\oval(0.23,0.23)[lt]}
\thinlines
\put(1, 4){\line(1,0){14}} \put(2.5,3.9){\line(0,1){0.2}}
\put(1.8,3.5){$s_{i-\kappa+1}$}
\put(3.5,3.9){\line(0,1){0.2}}
\put(3.4,3.5){$s_{i-\kappa+2}$}
\put(5.5,3.9){\line(0,1){0.2}}
\put(5.2,3.5){$s_{i-\kappa+3}$}
\put(8,3.9){\line(0,1){0.2}}
\put(7.9,3.5){$s_i$}
\put(13,3.9){\line(0,1){0.2}}
\put(12.6,3.5){$s_{i+1}$}
\put(2.6,3.85){$($}\put(13.3,3.85){$)$}
\put(2.8,4.2){$L$} \put(13,4.2){$R$}
\put(0,3.8){\vector(0,-1){1.6}}\put(0.2, 2.8){$\pi_{p+i}$}
\put(1, 2){\line(1,0){12}} \put(2.5,1.9){\line(0,1){0.2}}
\put(1.8,1.5){$c_{1-\kappa}$}
\put(3.5,1.9){\line(0,1){0.2}}\put(3.2,1.5){$c_{2-\kappa}$}
\put(5.5,1.9){\line(0,1){0.2}}\put(5.2,1.5){$c_{3-\kappa}$}
\put(8,1.9){\line(0,1){0.2}}\put(7.9,1.5){$c$}
\put(13.15, 1.875){\line(1,0){0.2}} \put(13.5,1.7){$c_1$}
\put(13,1.875){\oval(0.25,0.25)[r]}
\put(12.6, 1.75){\line(1,0){0.4}}
\put(0,1.8){\vector(0,-1){1.6}}\put(0.2, 0.8){$T$}
\put(8.2,1.4){\vector(4,-1){4.3}}
\put(2.5,1.2){\vector(1,-1){0.9}}
\put(1, 0){\line(1,0){12}}
\put(2.5, -0.1){\line(0,1){0.2}}\put(1.8,-0.5){$c_{1-\kappa}$}
\put(3.5, -0.1){\line(0,1){0.2}}
\put(5.5, -0.1){\line(0,1){0.2}}\put(5.2,0.3){$c_{3-\kappa}$}
\put(8, -0.1){\line(0,1){0.2}}\put(7.9,0.3){$c$}
\put(13.15,-0.125){\line(1,0){0.2}}\put(13.5,-0.3){$c_1$}
\put(4.5,-0.375){\oval(0.25,0.25)[l]}
\put(4.5,-0.5){\line(1,0){0.4}}
\put(4.25, -0.375){\line(1,0){0.1}}\put(3.8,-0.5){$c_2$}
\put(3.85,-0.7){$\underbrace{\qquad}_H$}
\end{picture}
\caption{The arc $J$ and its image under $\pi_{p+i}$ and $T \circ \pi_{p+i} = \pi_{p+i-1}$.}
\label{fig:pointsc-j}
\end{figure}
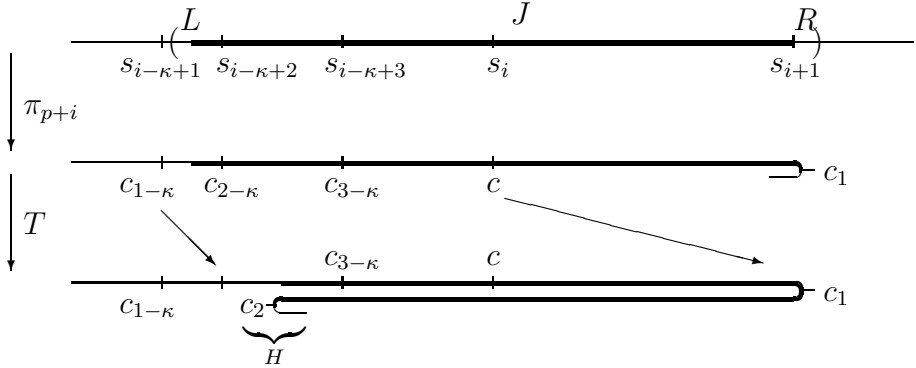

Let $J = [x, s_{i+1}]$ be such that $J \owns s_i$ and $\bar{d}(s_i, s_{i+1}) = \bar{d}(x, s_i)$,
where $\bar{d}$ is defined in Definition~\ref{df:dbar}. Then $\pi_{p+i-1}$ maps $J$ in
a $2$-to-$1$ fashion onto $[c_2,c_1]$, with $\pi_{p+i-1}(s_i) = c_1$ and $\pi_{p+i-1}(s_{i+1}) = c_2$.
Therefore $J$ is $p$-symmetric and also $p$-link-symmetric around $s_i$.
Since $c_{2-\kappa} < c_2 < c_{3-\kappa}$, we have $\pi_{p+i-1}(J) \not\owns c_{2-\kappa}$. Extend
$J$ on either side by equally long arcs $L$ and $R$ such that $\pi_{p+i-1}(L \cup R) = H$,
see Figure~\ref{fig:pointsc-j}. Since
$T^{p+i-1}|_H$ is not $\eps$-symmetric, $A_i \not \supset L \cup J \cup R$ provided
$\width(\chain_p) < \eps$. Hence $A_i \not\owns s_{i-\kappa+1}$ as claimed.
\end{proof}

\begin{rem} The bound $\kappa$ in this lemma is not sharp if $T_s$ has a periodic critical point.
For example, for the tent map with $c_2 < c = c_3 < c_1$, the folding pattern is
$$FP(\C0) = * \underbrace{ 0 \, \underline{1} \, 0 \, \underline{2} \, 0 \, 1 \, \underline{3} \,\overbrace{ 1 \, 0 \, 2 \, 0 \, \underline{4} \, 0 \, 2 \, 0 \, 1 \, 3 \, 1 \, 0 \, \underline{5} \,
0 \, 1 \, 3 \, 1 \, 0 \, 2 \, 0 \, 4 \, 0 \, 2 \, 0 \, 1 }^{\text{maximal $p$-symmetric}} \, \underline{6} \, 1 \, 0 \, 2 \, 0 \, 4 \, 0 \, 2}_{\text{maximal $p$-link-symmetric}} 0 \, 1 \dots$$
where $p$-levels of snappy $p$-points are underlined and $*$ denotes the conventional
$p$-level of $\alpha$.
Since $c$ has period $3$, so $c_a=c_{a+3b}$ for all $a,b \in \N$,
$p$-link-symmetric arcs can be longer than $p$-symmetric arcs.
Indeed, the maximal $p$-symmetric arc centered at snappy point $s_5$
stretches from $s_3$ to $s_6$, while maximal $p$-link-symmetric arc centered at
$s_5$ stretches almost from $\alpha$ to some point with $p$-level $2$.
This property holds for all snappy points: the
maximal $p$-link-symmetric arc around $s_i$ contains $s_j$ for all $j \leq i+1$.

A preperiodic example is $s = 2$, \ie $\ILs$ is the Knaster continuum.
\end{rem}

\begin{lem}\label{lem:snappy2}
Assume that $\width(\chain_p) < \eps$ and fix $i \in \N$, $i > \kappa-1$, and let $\ell^i$ and
$\ell^{i-1}$ be the links of $\chain_p$ containing $s_i$ and $s_{i-1}$ respectively.
Let $y$ be such that
$s_{i-1} \prec y \prec s_i$ and $y$ is not contained in the same arc-component
of $\ell^i$ as $s_i$, nor in the same arc-component of $\ell^{i-1}$
as $s_{i-1}$.
Then the maximal $p$-link-symmetric arc $J$ with center $y$ contains
at most one snappy $p$-point, and $J \subset A_i$.
\end{lem}
\begin{proof}
Let $\cdots < c_{-2} < c_{-1} < c_0 = c$ be the successive precritical points to
the left of $c$ with $T^j(c_{-j}) = c$. Since $A_i$ contains $s_{i+1}$ and its symmetric
point around $s_i$ (at least as boundary points), we have
$\pi_{p+i}(A_i) \supset [\hat c_1, c_1] \supset [c_{-1}, c]$. Let
$H := \pi_{p+i}(J)$ with center $x := \pi_{p+i}(y) \in [c_{-1}, c]$.
Assume by contradiction that $J$ contains two snappy $p$-points, or that
$J \not \subset A_i$. Then $|H| \geq 22\delta$ by the choice of $\delta$ in
\eqref{eq:delta}.

Let $w := (c_{-1} + c)/2$.
We distinguish four cases.
\begin{enumerate}
\item $c-\delta < x < c$. If $|T^{i}([x,c])| \leq \eps$, then we cannot `$\eps$-distinguish'
$x$ from $c$, violating our assumption that $y$ and $s_i$ do not belong to the
same arc-component of the same link.
If $|T^{i}([x,c])| > \eps$, then $T^{i}$ is $\eps$-symmetric on $H$ with centers $x$ and $c$, so
$T^{i}$ is $\eps$-periodic on $H$ with period $2|x-c|$. This leads
to a contradiction by the argument of the proof of Lemma~\ref{lem:eps2}.
\item $w \leq x \leq c-\delta$. Then $H$ satisfies \eqref{eq:H},
so by Proposition~\ref{prop:max_link_symmetry}, $T^{i}|_H$ cannot
be $\eps$-symmetric.
\item  $c_{-1}+\delta/s \leq x < w$. Then
by assumption $H$ contains one of $c$, $c_{-2}$
or $\hat c_1$ (whence $|H| \geq 22\delta$),
and hence $T(H \cap [c_{-2},c] \cap [\hat c_1, c])$
satisfies \eqref{eq:H}, so $T^{i}|_H$ cannot be $\eps$-symmetric by
Proposition~\ref{prop:max_link_symmetry}.
\item $c_{-1} < x < c_{-1}+\delta/s$. If $|T^{i}([c_{-1},x])| \leq \eps$,
then we cannot `$\eps$-distinguish' $x$ from $c$, violating the assumption that $y$ and
$s_{i-1}$ are not contained in the same arc component of $\ell^{i-1}$.
If $|T^{i}([c_{-1},x])| > \eps$ and again, $H$
by assumption contains one
of $c$, $c_{-2}$ or $\hat c_1$ (so $|H| \geq 22\delta$), then $T^{i-1}$ is
$\eps$-periodic on $T(H)$ which again leads to a contradiction by the argument of the proof of
Lemma~\ref{lem:eps2}.
\end{enumerate}
This proves the lemma.
\end{proof}

\section{Link-Symmetric Arcs and Homeomorphisms}\label{sec:link-symmetry}

In this section we study the action of homeomorphisms
$h:\ILsp \to \ILs$ on snappy $q$-points and $q$-points in general.
Let $q, p, g \in \N_0$ be such that
\[
h(\chain_{q}) \preceq \chain_{p} \preceq h(\chain_{g}).
\]
Recall that we assumed the slopes $s'$ and $s$ to be such that the critical points $c'$ and $c$ are not
(pre)periodic. Clearly $h$ maps the zero-composant $\C0'$ of $\ILsp$ to the zero-composant
$\C0$ of $\ILs$, and in particular the endpoint $\alpha'$ of $\C0'$ to the endpoint $\alpha$
of $\C0$. Let $\kappa' := \min\{ i \geq 3 : c'_i \geq c'\}$, where $c'_i = T_{s'}^i(c')$.
Let us denote the snappy $q$-points (\ie associated with $\chain_q$) by $s'_i$ and the snappy
$g$-points by $s''_i$. Therefore, snappy $q$-point $s'_i$ is the same as snappy $g$-point
$s''_{i+q-g}$. Similarly, let $A'_i$ be the maximal $q$-link-symmetric arc centered at $s'_i$
while as before, $A_i$ denotes the maximal $p$-link-symmetric arc centered at $s_i$

Since $A'_i$ is $q$-link-symmetric, and $h(\chain_{q}) \preceq \chain_{p}$,
the image  $D_i := h(A'_i) \subset \C0$ is $p$-link-symmetric and therefore
has a well-defined center, we denote it as $m_i$,
and a well-defined central link  $\ell_p$
(see Definition~\ref{df:linksym}).
In fact, $h(s'_i)$ and $m_i$ belong to the central link  $\ell_p$ and $m_i$ is the $p$-point
with the highest $p$-level of all $p$-points of the arc component of $\ell_p$ which contains $h(s_i)$.
Let $M_i := L_p(m_i)$.
\begin{theorem}\label{thm:shift}
$M_{i+1} = M_i + 1$ for all sufficiently large integers $i \in \N$.
\end{theorem}

\begin{proof}
Without loss of generality we can assume that $s' \ge s$,
so that $\kappa' \ge \kappa$.
We prove first that if $N \geq \kappa$ is so large that $m_N$ lies beyond the $\kappa$-th
snappy $p$-point of $\C0$, then $L_p(y) < M_N$, for every $y \in (\alpha, m_N)$; \ie $m_N$ is snappy.

Assume by contradiction that there exists $y \in (\alpha, m_N)$ such
that $L_p(y) \ge M_N$. By taking $L_p(y)$ maximal with this
property, we can assume that $y = s_{j-1} \prec m_N \prec s_j$ for
some $j > \kappa$. More precisely, $m_N$ is not contained in the same arc-component of the link
containing $s_{j-1}$ as $s_{j-1}$, and similarly for $s_j$.
Lemma~\ref{lem:snappy2} implies that $D_N$
contains at most one snappy $p$-point and that $D_N \subset A_j$.
Let us denote by $B$ the $p$-link-symmetric arc such that $s_j$ is the center of $B$,
$D_N \subset B \subset A_j$ and $\bd D_N \cap \bd B \ne \emptyset$ (see Figure~\ref{fig:lem:l}).
Since $\chain_{p} \preceq h(\chain_{g})$, the arc $B'' = \sigma^{q - g} \circ h^{-1}(B)$ is
$g$-link-symmetric and contains the arc
$\sigma^{q - g} \circ h^{-1}(D_N) = \sigma^{q - g}(A'_N)$.
The center $z''$ of $B''$ is the center of the arc component of
the central link $\ell_g$ of $B''$ containing
$\sigma^{q - g} \circ h^{-1}(s_j)$.
By Lemma~\ref{lem:snappy}, $A'_N$ contains $\kappa'$ snappy $q$-points
$s'_{N-\kappa'+2}, \dots ,s'_N, s'_{N+1}$.

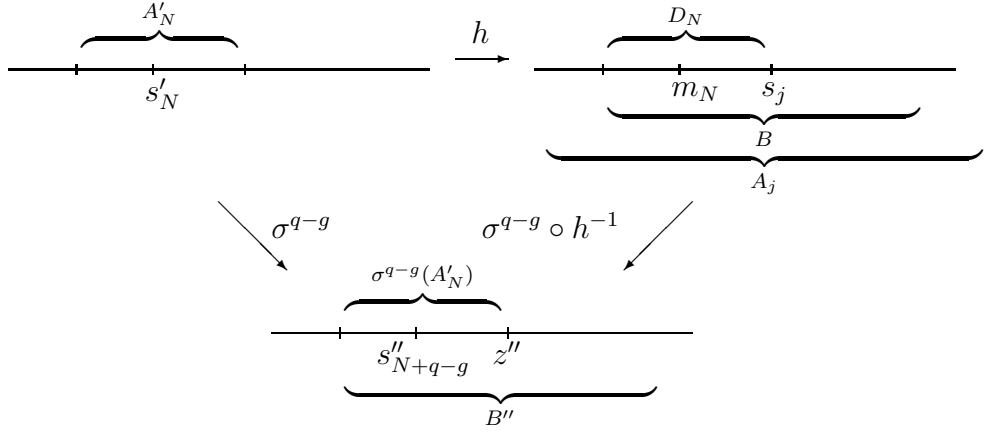
\begin{figure}[ht]
\unitlength=7mm
\begin{picture}(15,8)(1,-0.5)
\put(1,6){\line(1,0){8}}
\put(2.3,5.9){\line(0,1){0.2}}
\put(3.75,5.9){\line(0,1){0.2}}
\put(3.6,5.4){$s'_N$}
\put(2.4,6.4){$\overbrace{\qquad\qquad \quad}^{A'_N}$}
\put(5.5,5.9){\line(0,1){0.2}}
\put(9.5,6.2){\vector(1,0){1}}\put(9.8,6.5){$h$}
\put(11,6){\line(1,0){8}}
\put(12.3,5.9){\line(0,1){0.2}}
\put(13.75,5.9){\line(0,1){0.2}}\put(13.6,5.5){$m_N$}
\put(12.4,6.4){$\overbrace{\qquad\qquad \quad}^{D_N}$}
\put(15.5,5.9){\line(0,1){0.2}}\put(15.3,5.5){$s_j$}
\put(12.4,5.3){$\underbrace{\qquad \qquad \qquad \qquad \qquad}_{B}$}
\put(11.23,4.5){$\underbrace{\qquad \qquad \qquad\qquad \qquad \qquad \qquad}_{A_j}$}
\put(5,3.5){\vector(1,-1){1.3}}\put(6,2.8){$\sigma^{q-g}$}
\put(14,3.5){\vector(-1,-1){1.3}}\put(10,2.8){$\sigma^{q-g} \circ h^{-1}$}
\put(6,1){\line(1,0){8}}
\put(7.3,0.9){\line(0,1){0.2}}
\put(8.75,0.9){\line(0,1){0.2}}\put(8,0.4){$s''_{N+q-g}$}
\put(7.4,1.4){$\overbrace{\qquad\qquad \quad}^{\sigma^{q-g}(A'_{N})}$}
\put(10.5,0.9){\line(0,1){0.2}}\put(10.2,0.4){$z''$}
\put(7.4,0){$\underbrace{\qquad \qquad \qquad \qquad \qquad}_{B''}$}
\end{picture}
\caption{The relations between relative snappy points and arcs in $\chain_q$ (left),
$\chain_p$ (right), and $\chain_g$ (bottom).
\label{fig:lem:l}}
\end{figure}

The map $\sigma^{q-g}$ maps the $\kappa'$ snappy $q$-points $s'_i \in
A'_N$ to the $\kappa'$ snappy $g$-points $s''_{i+q-g} \in \sigma^{q-g}(A'_{N})$,
and $B''$ contains at least these $\kappa'$ snappy $g$-points. If
the center $z''$ of $B''$ is not snappy, then $B''$ contains at most one snappy $g$-point by
Lemma~\ref{lem:snappy2}, so we have a contradiction. Otherwise, if $z''$ is snappy, then even if
$z''$ is the right-most snappy $g$-point of $\sigma^{q-g}(A'_N)$, then still $B''$ contains $\kappa'-1$
snappy $g$-points on the left of the center $z''$, contradicting Lemma~\ref{lem:snappy}.
Therefore, $m_N$ is snappy.

Let us consider the arc $D_{N + \kappa' - 2} = h(A'_{N + \kappa' - 2})$.
Since $L_q(s'_{i+1}) - L_q(s'_i) = 1$, the arc $[s'_i, s'_{i+1}]$ contains
a $q$-point of every $q$-level less than $i$, so contains $q$-points of
$q$-levels 1 and 2. Therefore, $\pi_q([s'_i, s'_{i+1}]) = [c_2, c_1]$.
Note that two different points from $s'_N, \dots, s'_{N+\kappa'-1} \in A'_{N+\kappa'-2}$
can be mapped into the same link, say $\ell_p$ of $\chain_p$, but cannot
be mapped into the same arc component of $\ell_p$. Indeed, if $h([s'_i, s'_{i+1}]) \subset A$,
where $A$ is a arc component of $\ell_p$, then $h(\chain_q) \subset \ell_p$, a contradiction.
Therefore, $s_{M_N}, \dots , s_{M_{N + \kappa' -1}}$ are all different.

So, the arc $D_{N + \kappa' - 2}$ is $p$-link-symmetric and contains at least
$\kappa'$ snappy $p$-points, $s_{M_N}, \dots , s_{M_{N + \kappa' -1}}$.
By Lemma~\ref{lem:snappy}, the maximal $p$-link-symmetric arc
$A_{M_{N + \kappa' - 2}}$ centered at the snappy $p$-point $s_{M_{N + \kappa' -2}}$ contains $\kappa$ snappy $p$-points,
$s_{M_{N + \kappa' - 2} - \kappa + 2}, \dots , s_{M_{N + \kappa' - 2}}, s_{M_{N + \kappa' - 2} + 1}$.
Therefore, $D_{N + \kappa' - 2} \subseteq A_{M_{N + \kappa' - 2}}$, $\kappa' = \kappa$,
$s_{M_{N + i}} = s_{M_N + i}$ and $M_{N + i} = M_N + i$ for all $0
\leq i \leq \kappa-1$. By induction we get $M_{N + i} = M_N + i$ for
all $i \in \N_0$ as well.
\end{proof}

Every snappy $p$-point $s_i \in \C0$ can be contained in at most two
links of $\chain_p$, and one of them is always the central link of $A_i$,
which we will denote by $\ell_p^{s_i}$. Let $K_{s_i}$ be the arc component of
$\ell_p^{s_i}$ containing $s_i$. Given a $p$-point $x \in \C0$ with $L_p(x) = l$,
there can be two links of $\chain_p$ containing $x$, but one of them is always
$\ell_p^{s_l}$. We denote the arc component of $\ell_p^{s_l}$ containing $x$ by
$K_x$. Let $\ell_q^{s'_i} \in \chain_q$ and
$K_{s'_i} \subset \ell_q^{s'_i}$ be the similar notation related to $\C0'$ and $\chain_q$. Also, for a
$q$-point $x'$ of $\C0'$ with $L_q(x') = k$ let the arc component of $\ell_q^{s'_k}$ containing $x'$ be
denoted by $K_{x'}$.
\begin{pro}\label{pro:levels}
There exists $M \in \Z$ such that the following holds:
\begin{itemize}
\item[(1)] Let $l \in \N$ and let $x'$ be a $q$-point with $L_q(x') = l$. Then
$u = h(x') \in \ell_p^{s_{l+M}}$ and the arc component $K_u \subset \ell_p^{s_{l+M}}$ containing $u$,
also contains a $p$-point $x$ such that $L_p(x) = l + M$.
\item[(2)] For $l\in \N_0$ and $i\in \N$, the number of $q$-points in $[s'_i, s'_{i+1}]$ with
$q$-level $l$ is the same as the number of $p$-points in $[s_{M+i}, s_{M+i+1}]$ with $p$-level $M+l$.
\end{itemize}
\end{pro}
\begin{proof} $(1)$ Recall that the set of $q$-points in $\C0'$ is denoted by $E'_q$. By Theorem~\ref{thm:shift},
there exists $M \in \Z$ such that $a_i = h(s'_i) \in \ell_p^{s_{M+i}}$
for every $i \in \N_0$ and the arc
component $K_{a_i}$ of $\ell_p^{s_{M+i}}$ contains $s_{M+i}$. Therefore, statement $(1)$ is true for all snappy
$q$-points.

Also $h([s'_1, s'_2]) = [a_1,a_2]$, $s_{M+1} \in K_{a_1}$ and
$s_{M+2} \in K_{a_2}$. Let $q$-point $x'_1 \in [s'_2, s'_3]$ be such
that the arc $[s'_1,x'_1]$ is $q$-symmetric with center $s'_2$. Then
$h([s'_1, x'_1])$ is $p$-link-symmetric with center $s_{M+2}$. Since
there exists a unique $p$-point $b_1$ such that the arc
$[s_{M+1},b_1]$ is $p$-symmetric with center $s_{M+2}$, we have
$h(x'_1) \in K_{b_1}$, see Figure \ref{fig:levels}. Also $L_q(x'_1)
= 1$ and $L_p(b_1) = M+1$.

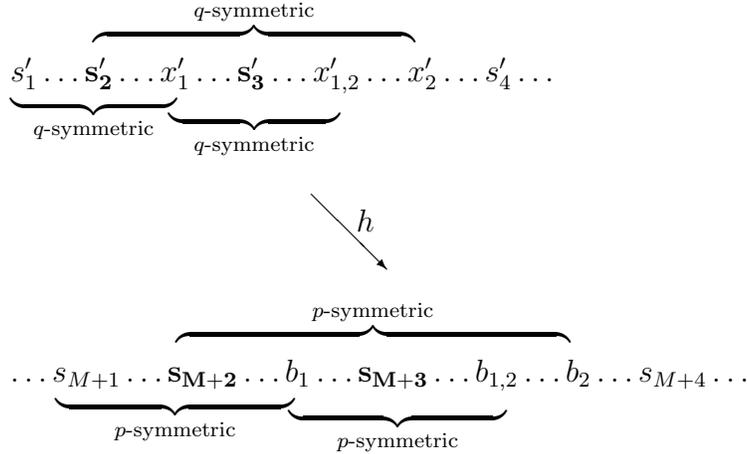
\begin{figure}[ht]
\unitlength=10mm
\begin{picture}(15,6.3)(-2,1)
\put(0,6){$s'_{1} \dots {\bf s'_2} \dots x'_1 \dots {\bf s'_{3}} \dots x'_{1,2} \dots x'_{2} \dots
s'_{4} \dots$}
\put(1.1,6.5){$\overbrace{\qquad \qquad \qquad \quad \quad \quad \quad \ }^{\textrm{$q$-symmetric}}$}
\put(0,5.8){$\underbrace{\qquad \qquad \quad \ }_{\textrm{$q$-symmetric}}$}
\put(2.1,5.6){$\underbrace{\qquad \qquad \quad \ \, }_{\textrm{$q$-symmetric}}$}
\put(4,4.5){\vector(1,-1){1}}\put(4.6,4){$h$}
\put(0,2){$\dots s_{M+1} \dots {\bf s_{M+2}} \dots b_1 \dots {\bf s_{M+3}} \dots b_{1,2} \dots b_{2} \dots s_{M+4} \dots $}
\put(2.2,2.5){$\overbrace{\qquad \qquad \qquad \qquad \qquad \quad \quad \ \ }^{\textrm{$p$-symmetric}}$}
\put(0.6,1.8){$\underbrace{\qquad \qquad \quad \quad \quad \ \ }_{\textrm{$p$-symmetric}}$}
\put(3.7,1.65){$\underbrace{\qquad \qquad \qquad \quad }_{\textrm{$p$-symmetric}}$}
\end{picture}
\caption{The configuration of symmetric arcs.
\label{fig:levels}}
\end{figure}

We have $h([s'_2, s'_3]) = [a_2,a_3]$, $s_{M+2} \in K_{a_2}$ and
$s_{M+3} \in K_{a_3}$. Let the $q$-point $x'_2 \in [s'_3, s'_4]$ be
such that the arc $[s'_2,x'_2]$ is $q$-symmetric with center $s'_3$.
Therefore $h([s'_2, x'_2])$ is $p$-link-symmetric with center $s_{M+3}$.
There exists a unique $p$-point $b_2$ such that the arc
$[s_{M+2},b_2]$ is $p$-symmetric with center $s_{M+3}$, so $h(x'_2)
\in K_{b_2}$. Also $L_q(x'_2) = 2$ and $L_p(b_2) = M+2$. Since
$[s'_2,x'_2]$ is $q$-symmetric, there exists a $q$-point $x'_{1,2}
\in [s'_3, x'_2]$ such that the arc $[x'_1,x'_{1,2}]$ is
$q$-symmetric with center $s'_3$. Then $h([x'_1,x'_{1,2}])$ is
$p$-link-symmetric with center $s_{M+3}$. Since there exists a
unique $p$-point $b_{1,2}$ such that the arc $[b_1,b_{1,2}]$ is
$p$-symmetric with center $s_{M+3}$, we have $h(x'_{1,2}) \in
K_{b_{1,2}}$, see Figure \ref{fig:levels}. Also $L_q(x'_{1,2}) = 1$
and $L_p(b_{1,2}) = M+1$.

The proof of $(1)$ follows by induction. Suppose at step $k$ we have
$h([s'_k, s'_{k+1}]) = [a_k,a_{k+1}]$, $s_{M+k} \in K_{a_k}$ and
$s_{M+k+1} \in K_{a_{k+1}}$, see Figure~\ref{fig:levels2}. Let again
$q$-point $x'_k \in [s'_{k+1}, s'_{k+2}]$ be such that the arc
$[s'_k, x'_k]$ is $q$-symmetric with center $s'_{k+1}$. Then
$h([s'_k, x'_k])$ is $p$-link-symmetric with center $s_{M+k+1}$. The
unique $p$-point $b_k$ such that $[s_{M+k},b_k]$ is $p$-symmetric
with center $s_{M+k+1}$ satisfies $h(x'_k) \in K_{b_k}$. Also
$L_q(x'_k) = k$ and $L_p(b_k) = M+k$.

\begin{figure}[ht]
\unitlength=10mm
\begin{picture}(13, 4.4)(0.5,0)
\put(1,3){\line(1,0){13}}
\put(2,2.9){\line(0,1){0.2}}\put(1.8,3.3){$s'_k$}
\put(4,2.9){\line(0,1){0.2}}\put(3.8,3.3){$s'_{k+1}$}
\put(5.5,2.9){\line(0,1){0.2}}\put(5.4,3.3){$x'$}
\put(6,2.9){\line(0,1){0.2}}\put(5.8,3.3){$x'_k$}
\put(6.5,2.9){\line(0,1){0.2}}\put(6.4,3.3){$y'$}
\put(7.5,2.9){\line(0,1){0.2}}\put(7.3,3.3){$s'_{k+2}$}
\put(11,2.9){\line(0,1){0.2}}\put(10.8,3.3){$x'_{k+1}$}
\put(13,2.9){\line(0,1){0.2}}\put(12.8,3.3){$s'_{k+3}$}
\put(5,3.8){$\overbrace{\quad \quad \ }^{\text{\small
$q$-symmetric}}$}
\put(1,2.6){\vector(0,-1){1.2}}\put(0.6,1.8){$h$}
\put(1,1){\line(1,0){13}}
\put(2,0.9){\line(0,1){0.2}}\put(1.8,1.3){$s_{M+k}$}
\put(1.6,0.5){$\approx h(s'_k)$}
\put(4,0.9){\line(0,1){0.2}}\put(3.8,1.3){$s_{M+k+1}$}
\put(5.5,0.9){\line(0,1){0.2}}\put(5.4,1.3){$x$}
\put(6,0.9){\line(0,1){0.2}}\put(5.9,1.3){$b_k$}
\put(6.5,0.9){\line(0,1){0.2}}\put(6.4,1.3){$y$}
\put(7.5,0.9){\line(0,1){0.2}}\put(7.3,1.3){$s_{M+k+2}$}
\put(7.2,0.5){$\approx h(s'_{k+2})$}
\put(11,0.9){\line(0,1){0.2}}\put(10.8,1.3){$b_{k+1}$}
\put(10.6,0.5){$\approx h(x'_{k+1})$}
\put(13,0.9){\line(0,1){0.2}}\put(12.8,1.3){$s_{M+k+3}$}
\put(12.6,0.5){$\approx h(s'_{k+3})$}
\put(5.1,0.8){$\underbrace{\quad \quad \ }_{\text{\small
$p$-link-sym.}}$}
\end{picture}
\caption{The relative point in the induction step. Here $\approx$
stands for ``belongs to the same arc component in the same link''.} \label{fig:levels2}
\end{figure}
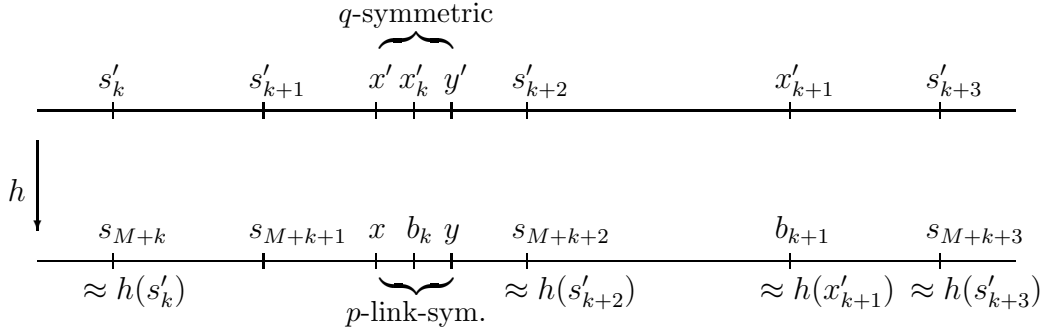

Let us suppose by induction that for every $q$-point $x' \in E_q$, $L_q(x') > 0$, $x' \prec x'_k$, we have
$u = h(x') \in \ell_p^{s_{r+M}}$, where $r = L_q(x')$, and the arc component $K_u \subset \ell_p^{s_{r+M}}$
contains a $p$-point $x$ such that $L_p(x) = r + M$. Since $L_q(x'_k) = k$, $L_q(s'_{k+1}) = k+1$ and
$L_q(s'_{k+2}) = k+2$, for every $q$-point $x' \in (s'_{k+1}, s'_{k+2})$, $x' \ne x'_k$, we have
$L_q(x') < L_q(x'_k)$. Hence for every $q$-point $y' \in (x'_{k}, s'_{k+2})$ there exists a $q$-point
$x' \in (s'_{k+1}, x'_{k})$ such that the arc $[x', y']$ is $q$-symmetric with center $x'_{k}$. So the arc
$h([x', y'])$ is $p$-link-symmetric with center $b_{k}$. The induction hypothesis implies that for $u = h(x')$,
the arc component $K_u \in \ell_p^{s_{r+M}}$ contains a $p$-point $x$ such that $L_p(x) = r + M$, where
$r = L_q(x')$.

Since $L_p(b_k) = M+k$, $L_p(s_{M+k+1}) = M+k+1$ and
$L_p(s_{M+k+2}) = M+k+2$,  we have $L_p(v) < L_p(b_k)$ for every
$p$-point $v \in (s_{M+k+1}, s_{M+k+2})$, $v \ne b_k$. Hence for
every $p$-point $v \in (b_{k}, s_{M+k+2})$ there exists a $p$-point
$w \in (s_{M+k+1}, b_{k})$ such that the arc $[w, v]$ is
$p$-symmetric with center $b_{k}$.
Therefore, and since $h([x', y'])$ is $p$-link-symmetric with center $b_{k}$,
there exists a unique $p$-point $y$ such that the arc $[x, y]$ is
$p$-symmetric with center $b_{k}$. Also, $h(y') \in K_{y}$ and
$L_p(y) = L_p(x)$, so $L_p(y) = L_q(y') + M$. This proves that for
every $q$-point $x' \in E_q$, $L_q(x') > 0$, $x' \prec s'_{k+2}$, we have
$u = h(x') \in \ell_p^{s_{r + M}}$, where $r = L_q(x')$, and the arc component
$K_u \subset \ell_p^{s_{r + M}}$ contains a $p$-point $x$ such that $L_p(x) = r + M$.

Next $h([s'_{k+1}, s'_{k+2}]) = [a_{k+1},a_{k+2}]$, $s_{M+k+1} \in
K_{a_{k+1}}$ and $s_{M+k+2} \in K_{a_{k+2}}$. Let the $q$-point
$x'_{k+1} \in [s'_{k+2}, s'_{k+3}]$ be such that the arc $[s'_{k+1},
x'_{k+1}]$ is $q$-symmetric with center $s'_{k+2}$. Then
$h([s'_{k+1}, x'_{k+1}])$ is $p$-link-symmetric with center
$s_{M+k+2}$. Since there exists a unique $p$-point $b_{k+1}$ such
that the arc $[s_{M+k+1},b_{k+1}]$ is $p$-symmetric with center
$s_{M+k+2}$, it follows that $h(x'_{k+1}) \in K_{b_{k+1}}$. Also,
$L_q(x'_{k+1}) = k+1$ and $L_p(b_{k+1}) = M+k+1$. Since $[s'_{k+1},
x'_{k+1}]$ is $q$-symmetric with center $s'_{k+2}$ and
$[s_{M+k+1},b_{k+1}]$ is $p$-symmetric with center $s_{M+k+2}$, the
same argument as above shows that for every $q$-point $x' \in E_q$,
$L_q(x') > 0$, $x' \prec x'_{k+1}$, we have $u = h(x') \in \ell_p^{s_{r + M}}$,
where $r=L_q(x')$, and the arc component $K_u \subset
\ell_p^{s_{r + M}}$ contains a $p$-point $x$ such that $L_p(x) = r + M$.
This proves the induction step.

$(2)$ Let $x$ be a $p$-point such that $L_p(x) > 0$ and $v = h^{-1}(x)$ lies beyond the $\kappa$-th
snappy $g$-point. Since $h^{-1}$ is also a homeomorphism and $h^{-1}(\chain_p) \prec \chain_g$, $(1)$
implies that there exists $M'$ such that $v \in \ell_g^{s''_{r + M'}}$, where $r = L_p(x)$. Also the
arc component $K_v \subset \ell_g^{s''_{r + M'}}$ contains a $g$-point $x''$ such that $L_g(x'') = r + M'$.

Let $x'$ be a $q$-point such that $L_q(x') > 0$, $x'$ lies beyond the $\kappa$-th snappy $g$-point and
$u = h(x')$ lies beyond the $\kappa$-th snappy $p$-point. Then $u \in \ell_p^{s_{r'+M}}$, where
$r' = L_q(x')$, and the arc component $K_u \subset \ell_p^{s_{r'+M}}$ contains a $p$-point $x$ such that
$L_p(x) = r' + M$. Also $v = h^{-1}(x) \in \ell_g^{s''_{r'+M+M'}}$ and the arc component
$K_v \subset \ell_g^{s''_{r'+M+M'}}$ contains a $g$-point $x''$ such that $L_g(x'') = L_q(x') + M + M'$.
Since $h^{-1} \circ h = id$, we have $x'' = x'$. Also $L_g(x'') =
L_q(x') + q-g$ implies that $M + M' = q-g$.
Since the number of $q$-points in $[s'_i, s'_{i+1}]$ with $q$-level $l$, $l \in \N_0$, is the same as the
number of $g$-points in $[s''_{q-g+i}, s''_{q-g+i+1}]$ with $g$-level $q-g+l$, it follows that this number
is the same as the number of $p$-points in $[s_{M+i}, s_{M+i+1}]$ with $p$-level $M+l$.
\end{proof}
\begin{proof}[Proof of Theorem~\ref{thm:Ingram}]
By \cite{Stim} we can assume that the critical points of $T_s$ and
$T_{s'}$ have infinite orbits. Therefore the above proposition shows that
$$
FP_q([s'_{k}, s'_{k+1}]) = FP_{p+M}([s_{M+k},s_{M+k+1}]) =
FP_p([s_{k},s_{k+1}]),
$$
for every positive integer $k$, and therefore
$FP(\C0') = FP(\C0)$, implying $s' = s$.
This proves the Ingram Conjecture.
\end{proof}

\section{Pseudo-isotopy}\label{sec:Pseudo}

Throughout this section, $h: \ILs \to \ILs$ will be an arbitrary self-homeomorphism.
We will extend Proposition~\ref{pro:levels} in order to prove the result on pseudo-isotopy.
Note that $(1)$ and $(2)$ of Proposition~\ref{pro:levels} together show that
$h$ induces an order preserving injection $h_{q, p}$ from
$E_{q}$ to $E_{p}$ such that $h_{q, p}(E_{q,i}) = E_{p,M+i} = E_{p+M,i}$ for every $i \in \N_0$,
where $E_{r,l}$ denotes the set of all $r$-points with $r$-level $l$ (see Definition~\ref{df:p-point}).
In fact $h_{q, p}$ is an order preserving bijection from $E_{q}$ to $E_{p+M}$ and is defined as follows:

\begin{df}\label{df:hqp}  Let $x \in E_q$. If $x = s'_i$ for some $i \in \N$, we define
$h_{q, p}(s'_i) = s_{M+i} \in E_p$. For all other $x\in E_q$, there exists $i \in \N$ such
that $x \in (s'_i, s'_{i+1})$. By Proposition~\ref{pro:levels}, the number of $q$-points of $(s'_i, s'_{i+1})$
is the same as the number of $(p+M)$-points of $(s_{M+i}, s'_{M+i+1})$. Let
$(s'_i, s'_{i+1}) \cap E_q = \{ x^0, \dots , x^n \}$ and
$(s_{M+i}, s'_{M+i+1}) \cap E_{p+M} = \{ y^0, \dots , y^n \}$. We define $h_{q, p}(x^i) = y^i$,
$i = 0, \dots , n$.
\end{df}

The next lemma shows that $h_{q,p}$ is essentially independent of $q$ and $p$.

\begin{lem}\label{lem:independent} If $q_1, p_1 \in \N$ are such that
$h(\chain_{q_1}) \prec \chain_{p_1} \prec h(\chain_q) \prec \chain_p$, then
$h_{q_1,p_1}|_{E_{q_1}} = h_{q,p}|_{E_{q_1}}$.
\end{lem}

\begin{proof}
By Proposition~\ref{pro:levels}, $h(\chain_q) \prec \chain_p$ implies that
there exists
$M \in \Z$ such that $h_{q,p}(E_{q,i}) = E_{p,M+i}$ for every $i \in \N_0$. Also,
$h(\chain_{q_1}) \prec \chain_{p_1}$ implies that there exists $M_1 \in \Z$ such that
$h_{q_1,p_1}(E_{q_1,i}) = E_{p_1,M_1+i}$ for every $i \in \N_0$.
Let $r, l \in \N$ be such that $q_1 = q + r$ and $p_1 = p + l$. Since $E_{q+r,i} = E_{q,r+i}$, we have
$$
h_{q,p}(E_{q+r,i}) = h_{q,p}(E_{q,r+i}) = E_{p,M+r+i},
$$
and also
$$
h_{q+r,p+l}(E_{q+r,i}) = E_{p+l,M_1+i} = E_{p,M_1+l+i}.
$$
We want to prove that $M+r = M_1+l$. To see this it suffices to pick a convenient point $x$ in
$E_{q+r,j}$ for some $j \in \N$, and to prove that $h_{q,p}(x) = y = h_{q+r,p+l}(x)$. Then the fact
that $y \in E_{p,M+r+j}$ and $y \in E_{p,M_1+l+j}$ implies that $M+r+j = L_p(y) = M_1+l+j$. For us,
the convenient choice of  $x \in E_{q+r} \subset E_q$ is a snappy $(q+r)$-point.

Let us denote the snappy $(q+r)$-points by $\hat s'_i$ and the snappy $(p+l)$-points by $\hat s_i$,
while as before $s'_i$ denotes the snappy $q$-points and $s_i$ denotes the snappy $p$-points. Note
that the snappy $(q+r)$-point $\hat s'_i$ is the same as the snappy $q$-point $s'_{i+r}$, and the
snappy $(p+l)$-point $\hat s_i$ is the same as the snappy $p$-point $s_{i+l}$. Let us denote the
maximal $(q+r)$-link-symmetric arc with the center $\hat s'_i$ by $\hat A'_i$, and the maximal
$(p+l)$-link-symmetric arc with the center $\hat s_i$ by $\hat A_i$, while as before $A'_i$ denotes the
maximal $q$-link-symmetric arc with the center $s'_i$, and $A_i$ denotes the maximal $p$-link-symmetric
arc with the center $s_i$. Note that $h(\hat A'_i) \subseteq \hat A_{M_1+i}$,
$h(A'_{i+r}) \subseteq A_{M+i+r}$ and $\hat s'_i = s'_{i+r}$. Also, the center of $\hat A_{M_1+i}$ is
$\hat s_{M_1+i} = s_{M_1+i+l}$ and the center of $A_{M+i+r}$ is $s_{M+i+r}$. Therefore,
$s_{M+i+r} = s_{M_1+i+l}$ and $M+r = M_1+l$.
\end{proof}

\begin{cor}\label{cor:independent} $R = M + p - q$ does not depend on $M, p, q$.
\end{cor}

\begin{proof} By Lemma \ref{lem:independent}, $M_1 + l = M + r$. Therefore
$R_1 = M_1 + p_1 - q_1 = M_1 + (p + l) - (q + r) = M + r + p - q - r = R$.
\end{proof}

\begin{df}\label{df:bridges}
We call an arc $B \in \C0$ a \emph{$p$-bridge} if the boundary points of $B$ are $p$-points with $p$-level
$0$, and if $L_p(x) \ne 0$ for every $p$-point $x \in \Int B$.
\end{df}

\begin{cor}\label{cor:bridges}
Let $B' \subset \C0$ be a $(q+1)$-bridge and $\partial B' = \{ a', b' \}$. There exists a $(p+M+1)$-bridge
$B$ such that for $\partial B = \{ a, b \}$ we have $h(B') \subset K_a \cup  B \cup K_b$ and
$h(a') \in K_a$, $h(b') \in K_b$, where $K_a$ and $K_b$ are the arc-components of the link $\ell_p^{s_{M+1}}$
of $\chain_p$ containing $a$ and $b$ respectively.
\end{cor}

\begin{proof}
Proposition~\ref{pro:levels} dealt with points in $E_{q,j}$ for $j \geq 1$,
but bridges involve points of level zero. Since $E_{q,1} = E_{q+1,0}$,
in this corollary we can work with $(q+1)$-bridges.

For each $j \geq 1$, $E_{q,j}$ is contained in a single link
$\ell^{s'_j}_q \in \chain_q$ and by Proposition~\ref{pro:levels}, for $\ell^{s_{M+j}}_p
\supseteq h(\ell^{s'_j}_q)$, every point of $h(E_{q,j})$ is contained in an
arc component of $\ell^{s_{M+j}}_p$ which contains a $p$-point of $E_{p,M+j} =
E_{p+M,j}$.
Since $E_{q+1,0} = E_{q,1}$ and $E_{p+M+1,0} = E_{p+M,1}$, every point of
$h(E_{q+1,0}) = h(E_{q,1})$ is contained in an arc component of
$\ell^{s_{M+1}}_{p}$ which contains a point of $E_{p+M,1} = E_{p+M+1,0}$.

Every two adjacent points of $E_{q+1,0}$ are the boundary points of a
$(q+1)$-bridge, and every two adjacent points of $E_{p+M+1,0}$ are the
boundary points of a $(p+M+1)$-bridge. We also have
$h_{q, p+M}(E_{q+1,0}) = h_{q, p+M}(E_{q,1}) = E_{p,M+1} = E_{p+M+1,0}$.
Therefore, for every $(q+1)$-bridge $B'$ there exists a $(p+M+1)$-bridge $B$ such that
$h_{q, p+M}(B') = B$. More precisely, for every
$(q+1)$-bridge $B'$ and $\partial B' = \{ a', b' \}$, there exists a
$(p+M+1)$-bridge $B$ such that for $\partial B = \{ a, b \}$ we have
$h(B') \subset K_a \cup B \cup K_b$ with $h(a') \in K_a$ and $h(b')
\in K_b$. Note that if $B'$ is a $(q+1)$-bridge with center $z'$ and
$\partial B' = \{ a', b' \}$ and $B'$ is contained in a single link
$\ell_{q+1}^{s'_1}$, then $h(B')$ is contained in the arc component
$K_a = K_b$ which contains also a $(p+M+1)$-point $z$ such that
$L_{p+M+1}(z) = L_{q+1}(z')$. So the arc component $K_a$ contains a
$(p+M+1)$-bridge $B$ with center $z$ and we have again $h(B')
\subset K_a \cup  B \cup K_b$.
\end{proof}

\begin{ex}\label{ex:sin1x}
A {\em $\sin \frac1x$-continuum} is a homeomorphic copy of
\[
\Big( \{ 0 \} \times [-1,1] \Big) \  \cup \ \Big\{ (x, \sin \frac1x) : x \in (0,1] \Big\}
\]
and the arc $\{ 0 \} \times [-1,1]$ is called the {\em bar} of the
$\sin \frac1x$-continuum.
Assume that $s>\sqrt{2}$ is such that the inverse limit $\ILs$ contains a $\sin \frac1x$-continuum $H$.
(Such $s$ exist in abundance, cf.\ \cite{BBD} and \cite{Bsubcontinua}.)
Then $\{\sigma^{-n}(H)\}_{n=0}^{\infty}$ is a sequence of pairwise disjoint $\sin \frac 1x$-continua with
$\diam(\sigma^{-n}(H))\to 0$ as $n\to \infty$.
There is then a sequence of disjoint neighborhoods $U_n$ of $\sigma^{-n}(H)$ with $diam(U_n)\to 0$. For each $n$, $U_n\cap \C0$ contains arbitrarily long arcs. Pick a sequence of arcs
$A_n \subset U_n \cap \C0$ of arc-length $\geq n+1$,
and construct a bijection $h: \ILs \circlearrowleft$ such that $h$ is the
identity on $\ILs \setminus \cup_n A_n$ and on each $A_n$,
$h$ fixes $\partial A_n$, but moves some points in
$A_n$ homeomorphically such that there is $x_n \in A_n$ with
$\bar d(x_n, h(x_n)) = n$.
Since $\diam(U_n) \to 0$, we find that $h$ is continuous and bijective.
Finally the compactness of $\ILs$ implies that
$h$ is a homeomorphism. Even though $h$ is isotopic to the identity,
$\sup_{x \in \C0} \bar d( x, h(x)) = \infty$.
\end{ex}

Therefore we cannot assume that a general self-homeomorphism of $\ILs$ has an $R \in \Z$ such that
$\sup_x \bar d(h(x), \sigma^R(x)) < \infty$. Block et al.\ \cite[Theorem 4.2]{BJK} used this property
to conclude that $h$ and $\sigma^R$ are pseudo-isotopic, \ie they permute the composants
of $\ILcores$
in the same way. However, since $\sigma^{-R} \circ h$ preserves
$(q+1)$-bridges for some $R \in \Z$ and $q$
sufficiently large, we can still follow the argument from \cite{BJK}.

\begin{proof}[Proof of Theorem~$\ref{thm:pseudo}$]
Let $P = s/(1+s) > 1/2$ be the orientation reversing fixed point of $T_s$ and $Q$ the center between $c_2$ and $c_1$.
 Let $\eps = \mesh (\chain_p)$ in Definition~\ref{df:chain}.
Without loss of generality, we can take $\eps/2 < \min\{ |c-P|, |c-Q| \}$.
Let $x \in \ILs \setminus \C0 = \ILcores$ be arbitrary. Recall that the composant of $x$ in $\ILcores$ is the
union of all proper subcontinua of $\ILcores$ containing $x$. Without loss of generality we can fix $q\in \N$
such that $\pi_{q+1}(x) \geq P$. Fix $p \in \N$ and $M \in \Z$ as in Proposition~\ref{pro:levels} such that
$h(\chain_q) \preceq \chain_p$ and $h$ sends $(q+1)$-bridges to $(p+M+1)$-bridges in terms of
Corollary~\ref{cor:bridges}. Let $R = M+p-q$, so $p+M+1 = q+R+1$.
Since by Corollary~\ref{cor:independent}, $R$ does not depend on $q$ and $p$, we can take  $q$ and $p$ larger
than $|R|$  without loss of generality.

Recall that the links $\ell^k_p$ of $\chain_p$ are of the form $\ell^k_p = \pi_{p}^{-1}(I^k_{p})$ of width $\leq \eps s^{-p}/2$.
The map $\sigma^{-R}$ maps the chain $\chain_p$ to a chain $\tilde \chain_{p-R}$ whose links are of the form
$\pi_{p-R}^{-1}(I^k_p)$ and hence also with width $\leq \eps s^{-p}/2$; this chain is coarser than
$\chain_{p-R}$ if $R < 0$. Furthermore, the $\sigma^{-R}$-image of a $(q+R+1)$-bridge is a $(q+1)$-bridge.

Take $\tilde h = \sigma^{-R} \circ h$. Since $h(\chain_q) \preceq \chain_p$, we have
$\tilde h(\chain_q) \preceq \tilde \chain_{p-R}$ and $\tilde h$ sends $(q+1)$-bridges to $(q+1)$-bridges, but
the `error' allowed in Corollary~\ref{cor:bridges}, \ie the arc-components of links from $\chain_p$, must now
be replaced by arc-components of links of $\tilde \chain_{p-R}$.
Recall that $\width(\chain_p) = \max_j |I^j_p|$, and
 $|\pi_{p-i}(\ell^j_p)| = |\pi_p(\ell^j_p)|s^i = |I^j_p|s^i$, for every
$0 \leq i \leq p$.
Therefore, $\pi_{p-R}(\tilde \ell^j_{p-R}) = \pi_{q-M}(\tilde \ell^j_{q-M}) \leq \eps s^{-p}/2$,
and
$\pi_{q+1}(\tilde \ell^j_{p-R}) = \pi_{q+1}(\tilde \ell^j_{q-M}) = \pi_{q-M}(\tilde \ell^j_{q-M})s^{-M-1} \leq \eps s^{-p-M-1}/2$. Thus, the $(q+1)$-th projection of links of $\tilde \chain_{p-R}$ are intervals of length
$\leq \eps s^{-(p+M+1)}/2 = \eps s^{-(q+R+1)}/2$, see Figure~\ref{fig:bridge}.

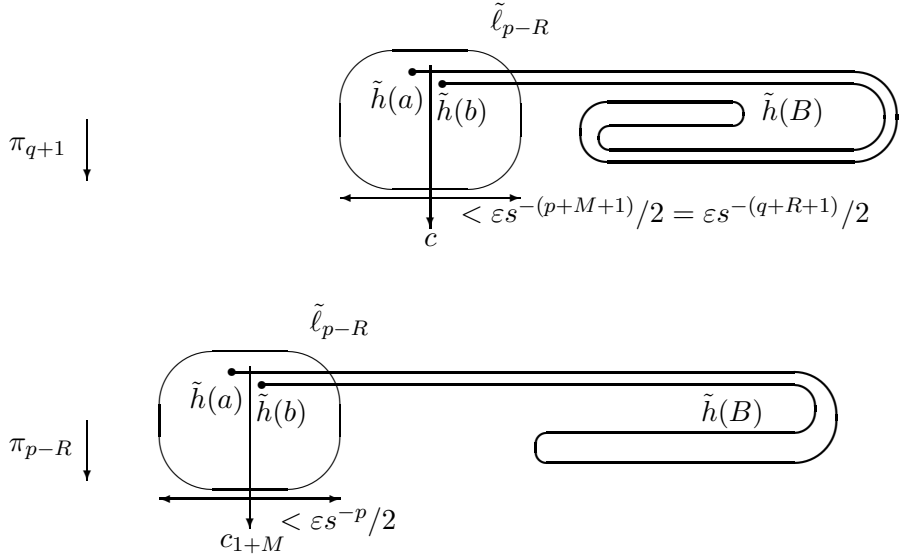
\begin{figure}[ht]
\unitlength=8mm
\begin{picture}(12,10)(2.5,0)
\thinlines
\put(10, 8){\oval(3,2.3)} \put(11,9.5){\small $\tilde \ell_{p-R}$}
\put(4.3,8){\vector(0,-1){1}} \put(3,7.5){\small $\pi_{q+1}$}
\put(10,8.9){\vector(0,-1){2.7}} \put(9.9,5.9){\small $c$}
\put(10,6.7){\vector(-1,0){1.5}} \put(10,6.7){\vector(1,0){1.5}}
\put(10.5,6.3){\small $< \eps s^{-(p+M+1)}/2 = \eps s^{-(q+R+1)}/2$}
\thicklines
\put(9.7,8.8){\line(1,0){7.3}}\put(9.7,8.8){\circle*{0.15}}\put(17, 8.05){\oval(1.5,1.5)[r]}
\put(10.2,8.6){\line(1,0){6.8}}\put(10.2,8.6){\circle*{0.15}}\put(17, 8.05){\oval(1.1,1.1)[r]}
\put(9, 8.2){\small $\tilde h(a)$}\put(10.1, 8){\small $\tilde h(b)$}
\put(17,7.3){\line(-1,0){4}} \put(13, 7.8){\oval(1,1)[l]}
\put(17,7.5){\line(-1,0){4}} \put(13, 7.7){\oval(0.4,0.4)[l]}
\put(13,8.3){\line(1,0){2}} \put(15, 8.1){\oval(0.4,0.4)[r]}
\put(13,7.9){\line(1,0){2}}
\put(15.5,8){\small $\tilde h(B)$}
\thinlines
\put(7, 3){\oval(3,2.3)}  \put(8,4.5){\small $\tilde \ell_{p-R}$}
\put(4.3,3){\vector(0,-1){1}} \put(3,2.5){\small $\pi_{p-R}$}
\put(7,3.9){\vector(0,-1){2.7}} \put(6.5,0.9){\small $c_{1+M}$}
\put(7,1.7){\vector(-1,0){1.5}} \put(7,1.7){\vector(1,0){1.5}}
\put(7.5,1.2){\small $< \eps s^{-p}/2$}
\thicklines
\put(6.7,3.8){\line(1,0){9.3}}\put(6.7,3.8){\circle*{0.15}}\put(16, 3.05){\oval(1.5,1.5)[r]}
\put(7.2,3.6){\line(1,0){8.8}}\put(7.2,3.6){\circle*{0.15}}\put(16, 3.2){\oval(0.8,0.8)[r]}
\put(6, 3.2){\small $\tilde h(a)$}\put(7.1, 3){\small $\tilde h(b)$}
\put(16,2.3){\line(-1,0){4}} \put(12, 2.55){\oval(0.5,0.5)[l]}
\put(16,2.8){\line(-1,0){4}}
\put(14.5,3){\small $\tilde h(B)$}
\end{picture}
\caption{The $(p-R)$-th and $(q+1)$-th projection of `the bridge' $\tilde h(B)$ with relevant link
$\tilde \ell_{p-R}$. The picture is suggestive of $M+1 \leq 0$;
if instead $M+1 > 0$, then
$\tilde h(B)$ contains fewer $(q+1)$-points than $(p-R)$-points.
}
\label{fig:bridge}
\end{figure}

The $(q+1)$-bridges that are small enough to belong to one or two links of $\chain_{q}$ will map to arcs contained in the link $\tilde \ell_{p-R}$.
Since $\pi_{q+1}(x) \geq P$ and $\eps s^{-(q+1)}/2 < |c-P|$,
no such short bridge can be close to $x$.
On the longer $(q+1)$-bridges of $\chain_q$ that map outside of
$\tilde \ell_{p-R}$, $\tilde h$ acts as a trivial one-to-one correspondence, sending the first such bridge
to the first, the second to the second, etc.

Find a sequence $(x_n)_{n \in \N} \subset \C0$ such that $x_n \to x$.
Then for large $n$, $x_n$ belongs to
a long $(q+1)$-bridge, and by the above argument, $\tilde h(x_n)$ and $x_n$ belong to the same $(q+1)$-bridge
up to an `error' of at most $\eps s^{-(q+R+1)}/2$.
Take $H_n = [\tilde h(x_n), x_n]$ and a subsequence
such that $H_{n_j} \to H$ in Hausdorff topology. Clearly $H$ is a continuum and $x, \tilde h(x) \in H$.
Since $\pi_{q+1}(x) \geq P$, the arcs $H_{n_j}$ belong to arcs whose $(q+1)$-projections belong to
$[c-\eps s^{-(q+R+1)}/2,c_1]$ for all sufficiently large $j$. Since $q+R+1 \ge 1$ and $\eps/2 < c - Q$, we have
$Q < c - \eps/2 < c - \eps s^{-(q+R+1)}/2$ implying $[c-\eps s^{-(q+R+1)}/2,c_1] \subset [Q, c_1]$.

Therefore
$\pi_{q+1}(H_{n_j}), \pi_{q+1}(H) \subset  [Q, c_1]$, and since $[Q, c_1]$ is a proper subset of $[c_2,c_1]$ and the inclusion holds for arbitrarily large $q$,
$H$ is a proper subcontinuum of
$\ILcores$. It follows that $\tilde h(x)$ and $x$ belong to the same composant of $\ILcores$. Apply $\sigma^R$ to
find that $h(x)$ and $\sigma^R(x)$ belong to the same composant as well.
\end{proof}

Pseudo-isotopy of $h$ implies that the number of composants being mapped to themselves
is the same for $h^n$ and $\sigma^{nR}$. This number grows like $s^{nR}$,
which in \cite{BJKK} provides a proof of the Ingram conjecture
for tent maps with periodic critical point.
In this situation, \cite{BJKK} in fact also shows that $h$ is isotopic to
a power of the shift. Due to the existence of composants
that are not arc-connected, this is not so clear in the general case.

\begin{rem}
Not every pseudo-isotopy is an isotopy. For instance, a
homeomorphism flipping the bar of a $\sin \frac1x$-continuum
cannot be isotopic to the identity.
If the bonding map is a quadratic map within the first period doubling
cascade, then the inverse limit space is a finite collection of
$\sin \frac1x$-continua, see \cite{BI}, and we can indeed construct
homeomorphism that are pseudo-isotopic but not isotopic to the identity.
Among those tent maps $T_s$, $s \in [\sqrt2,2]$,
whose inverse limit space is known to contain $\sin \frac1x$-continua,
both in \cite{BBD} and \cite{Bsubcontinua}, the topology is
much more complicated, as more than a
single ray can be expected to accumulate on their bars.
Thus the following question is very relevant:
\begin{quote}
Is every self-homeomorphism of $\ILs$
isotopic to a power of the shift?
\end{quote}
We know this to be true if $c$ is periodic or non-recurrent
\cite{BJKK, BKRS}, but this case is simpler, because
the only proper subcontinua of $\ILcores$ are arcs or points.
\end{rem}

\medskip
\noindent
Department of Mathematical Sciences\\
Montana State University\\
Bozeman, MT 59717, USA\\
\texttt{umsfmbar@math.montana.edu}\\
\texttt{http://www.math.montana.edu/}$\sim$\texttt{umsfmbar/}

\medskip
\noindent
Department of Mathematics\\
University of Surrey\\
Guildford, Surrey, GU2 7XH, UK\\
\texttt{h.bruin@surrey.ac.uk}\\
\texttt{http://personal.maths.surrey.ac.uk/st/H.Bruin/}

\medskip
\noindent
Department of Mathematics\\
University of Zagreb\\
Bijeni\v cka 30, 10 000 Zagreb, Croatia\\
\texttt{sonja@math.hr}\\
\texttt{http://www.math.hr/}$\sim$\texttt{sonja}

\end{document}